%% file: add_final.tex
\documentclass{amsart}

\input{decls}

\tikzset{lab/.style={auto,font=\scriptsize}} 
\usetikzlibrary{arrows}
\usepackage{eucal}
\usepackage{fullpage}

\autodefs{\cDER\cCat\cGrp\cCAT\cMONCAT\bSet\cProf\bCat
\cPro\cMONDER\cBICAT\ncomp\nid\sProf\ncyl\fC\fZ\fT\nhocolim\nholim\cPsNat\ncoll\bsSet\cAlg}

\let\oldboxtimes\boxtimes
\def\boxtimes{\mathrel{\oldboxtimes}}

\def\ccsub{_{\mathrm{cc}}}
\def\pdh(#1,#2){\llbracket #1,#2\rrbracket}
\def\ldh(#1,#2){\llbracket #1,#2\rrbracket\ccsub}
\def\pend(#1){\pdh(#1,#1)}
\def\lend(#1){\ldh(#1,#1)}

\def\DTl#1#2#3#4#5#6#7{%
  \xymatrix@C=3pc{{#1} \ar[r]^-{#2} &
    {#3} \ar[r]^-{#4} &
    {#5} \ar[r]^-{#6} &
    {#7}
  }}

\newsavebox{\tvabox}
\savebox\tvabox{\hspace{1mm}\begin{tikzpicture}[>=latex',baseline={(0,-.18)}]
  \draw[->] (0,.1) -- +(1,0);
  \node at (.5,0) {$\scriptscriptstyle\bot$};
  \draw[->] (1,-.1) -- +(-1,0);
  \draw[->] (1,-.2) -- +(-1,0);
\end{tikzpicture}\hspace{1mm}}

\renewcommand{\Set}{\mathrm{Set}}

\newcommand{\Ab}{\mathrm{Ab}}
\newcommand{\Alg} {\mathrm{Alg}}
\newcommand{\Alggp} {\Grp}

\newcommand{\Mon} {\mathrm{Mon}_{\Eoo}\!}
\newcommand{\Grp} {\mathrm{Grp}_{\Eoo}\!}
\newcommand{\Mono} {\mathrm{Mon}_{\Eoo}^\otimes\!}
\newcommand{\Grpo} {\mathrm{Grp}_{\Eoo}^\otimes\!}

\newcommand{\Mod} {\mathrm{Mod}}

\newcommand{\coAlg} {\mathrm{coMon}_{\Eoo}\!}

\newcommand{\coGrp} {\mathrm{coGrp}_{\Eoo}\!}

\newcommand{\calc} {\mathcal C}
\newcommand{\cald} {\mathcal D}
\newcommand{\cale} {\mathcal E}
\newcommand{\calf} {\mathcal F}
\newcommand{\cals} {\mathcal{S}}
\newcommand{\Fun} {\mathrm{Fun}}

\newcommand{\NFin} {\mathrm{N}(\mathcal{F}\mathrm{in}_*)}

\newcommand{\Rig} {\mathcal{R}\mathrm{ig}}
\newcommand{\Ring} {\mathcal{R}\mathrm{ing}}
\newcommand{\Prl}   {\mathcal{P}\mathrm{r^L}}
\newcommand{\Prr}   {\mathcal{P}\mathrm{r^R}}
\newcommand{\Prlpt}   {\mathcal{P}\mathrm{r^L_{Pt}}}
\newcommand{\Prlpre}   {\mathcal{P}\mathrm{r^L_{Pre}}}
\newcommand{\Prladd}   {\mathcal{P}\mathrm{r^L_{Add}}}
\newcommand{\Prlst}   {\mathcal{P}\mathrm{r^L_{St}}}
\newcommand{\Prlo}   {\mathcal{P}\mathrm{r^{L,\otimes}}}
\newcommand{\Prlpto}   {\mathcal{P}\mathrm{r^{L,\otimes}_{Pt}}}
\newcommand{\Prlpreo}   {\mathcal{P}\mathrm{r^{L,\otimes}_{Pre}}}
\newcommand{\Prladdo}   {\mathcal{P}\mathrm{r^{L,\otimes}_{Add}}}
\newcommand{\Prlsto}   {\mathcal{P}\mathrm{r^{L,\otimes}_{St}}}
\newcommand{\FunR}{\Fun^\mathrm{R}}
\newcommand{\FunRAd}{\Fun^\mathrm{RAd}}

\newcommand{\FunL}{\Fun^\mathrm{L}}
\newcommand{\FunLL}{\Fun^\mathrm{L,L}}
\newcommand{\FunLMon}{\Fun^{\mathrm{L},\otimes}}
\newcommand{\Eoo}{\mathbb{E}_\infty}
\newcommand{\Catoo}{\mathcal{C}\mathrm{at}_\infty}
\newcommand{\Cat}{\mathcal{C}\mathrm{at}}

\newcommand{\Sp}{\mathrm{Sp}}
\renewcommand{\Th}{\mathcal{T}\mathrm{h}}
\newcommand{\n}{\langle n\rangle}
\newcommand{\one}{\langle 1\rangle}

\renewcommand{\adj}{\rightleftarrows}
\newcommand{\SymMonCat}{\mathcal{S}\mathrm{ym}\mathcal{M}\mathrm{on}\mathcal{C}\mathrm{at}}
\newcommand{\SymMonCatoo}{\mathcal{S}\mathrm{ym}\mathcal{M}\mathrm{on}\mathcal{C}\mathrm{at}_\infty}

\newcommand{\calo}{\mathcal{O}}
\newcommand{\calp}{\mathcal{P}}
\newcommand{\bbE}{\mathbb{E}}
\newcommand{\RingSp}{\Ring\Sp}
\newcommand{\RigoCat}{\Rig_\calo\Cat}
\newcommand{\RingoCat}{\Ring_\calo\Cat}

\newcommand{\bbT}{\mathbb{T}}
\renewcommand{\id}{\mathrm{id}}
\renewcommand{\Ho}{\mathrm{Ho}}
\newcommand{\N}{\mathrm{N}}
\newcommand{\Fr}{\mathrm{Fr}}
\newcommand{\Proj}{\mathrm{Proj}}
\newcommand{\K}{\mathrm{K}}
\newcommand{\calm}{\mathcal{M}}
\newcommand{\ev}{\mathrm{ev}}

\newcommand{\Ind}{\mathrm{Ind}}

\newtheorem*{thmintro}{$\mathbf{Theorem}$}

\renewcommand{\S}{Section }

\title[]{Universality of multiplicative infinite loop space machines}
\author{David Gepner, Moritz Groth and Thomas Nikolaus}
\date{\today}
\begin{document}

\begin{abstract}
We establish a canonical and unique tensor product for commutative monoids and groups in an $\infty$-category $\calc$ which generalizes the ordinary tensor product of abelian groups. Using this tensor product we show that $\bbE_n$-(semi)ring objects in $\calc$ give rise to $\bbE_n$-ring spectrum objects in $\calc$. In the case that $\calc$ is the $\infty$-category of spaces this produces a multiplicative infinite loop space machine which can be applied to the algebraic K-theory of rings and ring spectra. 

The main tool we use to establish these results is the theory of smashing localizations of presentable $\infty$-categories. 
In particular, we identify preadditive and additive $\infty$-categories as the local objects for certain smashing localizations.
A central theme is the stability of algebraic structures under basechange; for example, we show $\Ring(\cald\otimes\calc)\simeq\Ring(\cald)\otimes\calc$.
Lastly, we also consider these algebraic structures from the perspective of Lawvere algebraic theories in $\infty$-categories.
\end{abstract}

\maketitle
\setcounter{section}{-1}

\tableofcontents

\section{Introduction}\label{sec:intro}

The Grothendieck group $\mathrm{K}_0{(M)}$ of a commutative monoid $M$, also known as the group completion, is the universal abelian group which receives a monoid map from $M$. It was a major insight of Quillen that higher algebraic K-groups can be defined as the homotopy groups of a certain spectrum which admits a similar description: more precisely, from the perspective of higher category theory, the algebraic K-theory spectrum of a ring $R$ can be understood as the group completion of the groupoid of projective $R$-modules, viewed as a symmetric monoidal category with respect to the coproduct.

When $R$ is commutative, the algebraic K-groups inherit a multiplication which stems from the tensor product of $R$-modules. Just as the K-groups arise as homotopy groups of the K-theory spectrum, it is essential for computational and theoretical purposes to understand the multiplication on these groups as coming from a highly structured multiplication on the K-theory spectrum itself. 
Unfortunately it turned out to be hard to construct such a multiplication directly, partly because for a long time the proper framework to deal with multiplicative structures on spectra was missing. Important work on this question was pioneered by May et.~al. \cite{May82}, and the general theory of homotopy coherent algebraic structures goes back at least to Boardman-Vogt \cite{BV}, May \cite{Operad}, and Segal \cite{segal_categories}.

It was first shown by May that the group completion functor from $\bbE_\infty$-spaces to spectra preserves multiplicative structure \cite{May82}; see also the more recent accounts \cite{May_WhatI, May_WhatII, May_WhatIII}. 
Since then, several authors have given alternative constructions of multiplicative structure on K-theory spectra: most notably, Elmendorf and Mandell promote the infinite loop space machine of Segal to a multifunctor in \cite{EM06} and in \cite{EM09} they extend the K-theory functor from symmetric monoidal categories to symmetric multicategories (a.k.a.\ coloured operads), and Baas-Dundas-Richter-Rognes show how to correct the failure of the `phony multiplication' on the Grayson-Quillen $S^{-1}S$-construction in \cite{BDRR}, as identified by Thomason \cite{phony}.

All of these approaches are very carefully crafted and involve for example the intricacies of specific pairs of operads or indexing categories. 
Here we take a different approach to `multiplicative infinite loop space theory', replacing the topological and combinatorial constructions of specific machines by the use of \emph{universal properties}. 
 The main advantage of our approach is that we get strong uniqueness results, which follow for free from the universal properties. 
The price we pay is that we use the extensive machinery of $\infty$-categories and argue in the abstract, without the aid of concrete models.
Similar results for the case of Waldhausen K-theory, also using the language of $\infty$-categories, have been obtained by Barwick in the
recent paper \cite{Bar}.\\

In this paper we choose to use the language of (presentable) $\infty$-categories. 
But we emphasize the fact that every combinatorial model category gives rise to a presentable $\infty$-category, and that
all presentable $\infty$-categories arise in this way.
Moreover the study of presentable $\infty$-categories is basically the same as the study of combinatorial model categories, so that in principle all our results could also be formulated in the setting of model categories.

Let us begin by mentioning one of our main results. Associated to an $\infty$-category~$\calc$ are the $\infty$-categories $\calc_*$ of pointed objects in $\calc$, $\Mon(\calc)$ of commutative monoids in $\calc$, $\Grp(\calc)$ of commutative groups in $\calc$, and $\Sp(\calc)$ of spectrum objects in $\calc$. For these $\infty$-categories we establish the following:

\begin{thmintro}(\autoref{thm:tensorproducts})
{\it
Let $\calc^\otimes$ be a closed symmetric monoidal structure on a presentable $\infty$-category~$\calc$.
The $\infty$-categories $\calc_*$, $\Mon(\calc)$, $\Grp(\calc)$, and $\Sp(\calc)$ all admit closed symmetric monoidal structures, which are uniquely determined by the requirement that the respective free functors from~$\calc$ are symmetric monoidal. Moreover, each of the following free functors also extends uniquely to a symmetric monoidal functor
\[
\calc_*\to\Mon(\calc) \to \Grp(\calc) \to \Sp(\calc)\, .
\]}
\end{thmintro}
\noindent
Note that these symmetric monoidal structures allow us to talk about $\bbE_n$-(semi)ring objects and $\bbE_n$-ring spectrum objects in $\calc$.
Before we sketch the general ideas involved in the proof, it is worth indicating what this theorem amounts to for specific choices of $\calc$. 
\begin{enumerate}
\item
 If $\calc$ is the ordinary category of sets, then the symmetric monoidal structures of \autoref{thm:tensorproducts} recover for instance the tensor product of abelian monoids and abelian groups. 
This also reestablishes the easy result that the group completion functor $\mathrm{K}_0$ is symmetric monoidal.

\item
In the case of the 2-category $\Cat$ of ordinary categories, functors, and natural isomorphisms we obtain a symmetric monoidal structure on the $2$-category of symmetric monoidal categories.  
The symmetric monoidal structure on $\SymMonCat \simeq \Mon(\Cat)$ has been the subject of confusion in the past due to the fact that $\SymMonCat$ only has the 
desired symmetric monoidal structure when considered as a 2-category and not as a 1-category. 
In this case, $\bbE_n$-(semi)ring objects are \emph{$\bbE_n$-(semi)ring categories} (sometimes also called \emph{rig categories}), important examples of which are given by the bipermutative categories of \cite{May_WhatII}. 
We also obtain higher categorical analogues of this picture using $\Cat_n$ and~$\Catoo$.\footnote{Interestingly, we have equivalences $\Grp(\Cat_n)\simeq\Grp(\mathrm{Gpd}_{n})$ and $\Sp(\Cat_n) \simeq \Sp(\mathrm{Gpd}_{n})$, and the latter is trivial unless $n = \infty$; more generally, $\Sp(\calc)$ is trivial for any $n$-category $\calc$ if $n$ is finite.}

\item
Finally, and most importantly for this paper, we consider \autoref{thm:tensorproducts} in the special case of the $\infty$-category~$\cals$ of spaces (which can be obtained from the model category of spaces or simplicial sets). That way we get canonical monoidal structures on $\bbE_\infty$-spaces and grouplike $\bbE_\infty$-spaces. The resulting $\bbE_n$-algebras are \emph{$\bbE_n$-(semi)ring spaces}; more precisely, they are an $\infty$-categorical analogue of the $\bbE_n$-(semi)ring spaces of May (see, for example,~\cite{May_WhatI}). Moreover, we obtain unique multiplicative structures on the group completion functor $\Mon(\cals) \to \Grp(\cals)$ and the `delooping' functor $\Grp(\cals) \to \Sp$ which assigns a spectrum to a grouplike $\bbE_\infty$-space. In particular, the spectrum associated to  an $\bbE_n$-(semi)ring space is an $\bbE_n$-ring spectrum, which amounts to \emph{`multiplicative infinite loop space theory'}. 
 \end{enumerate}

These facts can be assembled together to obtain a new description of the multiplicative structure on the algebraic K-theory functor $\K\colon \SymMonCat \to \Sp$  and its $\infty$-categorical variant  $\K\colon \SymMonCat_\infty \to \Sp$  (\S\ref{sec:infinite}). In particular, the algebraic K-theory of an $\bbE_n$-semiring ($\infty$-)category is canonically an 
 $\bbE_n$-ring spectrum. By a `recognition principle' for $\bbE_n$-semiring ($\infty$-)categories, this applies to many examples of interest. More precisely, we show that these semiring $\infty$-categories can be obtained from $\bbE_n$-monoidal $\infty$-categories with coproducts such that the monoidal structure preserves coproducts in each variable separately (\autoref{seminringex}). For instance ordinary closed monoidal, braided monoidal, or symmetric monoidal categories admit the structure of $\bbE_n$-semiring categories (for $n=1,2,\infty$, 
respectively) in which the `addition' is given by the coproduct and the `mutliplication' is given by the tensor product. More specific 
examples 
are given by ($\infty$-)categories of modules over ordinary commutative rings or  $\bbE_n$-ring spectra.\footnote{But note that the $\infty$-category of  modules for an $\bbE_n$-ring spectrum is only an $\bbE_{n-1}$-semiring $\infty$-category.}\\

One central idea to prove \autoref{thm:tensorproducts} as stated above, which is also of independent interest, is to identify the assignments 
\begin{equation}\label{assignments}
\calc \mapsto \calc_*, \qquad \calc \mapsto \Mon(\calc),\qquad \calc \mapsto \Grp(\calc),\qquad \calc \mapsto \Sp(\calc)
\end{equation}
 as universal constructions. The first and the last case have already been thoroughly
discussed by Lurie in~\cite{HA}, where it is shown that, in the world of presentable $\infty$-categories, 
$\calc_*$ is the \emph{free pointed} $\infty$-category on~$\calc$ and $\Sp(\calc)$ is the \emph{free stable} $\infty$-category on $\calc$. We extend this picture by introducing \emph{preadditive} and \emph{additive} $\infty$-categories (see also~\cite{TV} and \cite{Joyal}). These notions are obtained by imposing additional exactness conditions on pointed $\infty$-categories, just as is done in the case of ordinary categories. In fact, a presentable $\infty$-category $\calc$ is (pre)additive if and only if its homotopy category $\Ho(\calc)$ is (pre)additive in the sense of ordinary category theory. We show that, again in the framework of presentable $\infty$-categories, $\Mon(\calc)$ is the \emph{free preadditive} $\infty$-category on $\calc$ and that $\Grp(\calc)$ is the \emph{free additive} $\infty$-category on $\calc$ (\autoref{cor:frepre}). 

As an application of this description as free categories one can deduce the existence and uniqueness of the functors
\[
\calc\to\calc_*\to\Mon(\calc) \to \Grp(\calc) \to \Sp(\calc_\ast)
\]
from the fact that every stable $\infty$-category is additive, every additive $\infty$-category is preadditive and every preadditive $\infty$-category is pointed. 
More abstractly, the assignments \eqref{assignments} give rise to endofunctors of the $\infty$-category $\Prl$ of presentable $\infty$-categories and left adjoint functors. The aforementioned universal properties are equivalent to the observation that these endofunctors are localizations (in the sense of Bousfield) of $\Prl$ with local objects the pointed, preadditive, additive, and stable presentable $\infty$-categories, respectively. \\

A second main theme of the paper  is the stability of algebraic structures under basechange. For example we show that we have equivalences
$$
\Mon(\calc \otimes \cald) \simeq \Mon(\calc) \otimes \cald \qquad \Ring_{\bbE_n}(\calc \otimes \cald) \simeq  \Ring_{\bbE_n}(\calc) \otimes \cald\ ,
$$
where $\otimes$ denotes the tensor product on $\Prl$ as constructed in \cite{HA} (\autoref{cormon} and \autoref{prop_rig_ten}). Such basechange properties are satisfied by many endofunctors of $\Prl$ which arise when considering algebraic structures of certain kinds, e.g.\ $\calc \mapsto \Mod_\bbT(\calc)$ for a Lawvere algebraic theory $\bbT$. We give a brief account of algebraic theories in ~\autoref{sec:app}. 

A key insight here is to consider endofunctors of $\Prl$ which satisfy both properties: namely, they are simultaneously localizations and satisfy basechange. 
In keeping with the terminology of stable homotopy theory we refer to such functors as \emph{smashing localizations} of $\Prl$.  The endofunctors $(-)_\ast, \Mon, \Grp$ and $\Sp$ \eqref{assignments} are the main examples treated in this paper.
Then the proof of \autoref{thm:tensorproducts} follows as a special case of  the general theory of smashing  localizations $L\colon\Prl \to \Prl$. For example we prove that 
if $\calc \in \Prl$ is closed symmetric monoidal, then 
the $\infty$-category $L\calc$ admits a {\em unique} closed symmetric monoidal structure such that the localization map $\calc \to L\calc$ is a symmetric monoidal functor (\autoref{prop:smmonoidal}). \\

\noindent
\textbf{Organization of the paper.}
In \S\ref{sec:monoids}, we recall the definition of the $\infty$-category of monoid and group objects in an $\infty$-category. 
They form the generic examples of (pre)addtive $\infty$-categories which we 
introduce in \S\ref{sec:pre}. In \S\ref{sec:smash}, we study smashing localizations of $\Prl$, which turns out to be the central notion needed to deduce many of the subsequent results in this paper. We then show, in \S\ref{sec:mon}, that the formation of commutative monoids and groups in presentable $\infty$-categories are examples of smashing localizations of $\Prl$, and we identify these localizations with the free (pre)additive $\infty$-category functor. This leads to the existence of the canonical symmetric monoidal structures described in 
\S\ref{sec:spec}, and the next  
\S\ref{sec:more} is devoted to studying the functoriality of these structures. Then in  \S\ref{sec:ring} we consider $\infty$-categories of (semi)ring objects in a closed symmetric monoidal presentable $\infty$-category; these are used in \S\ref{sec:infinite} to show that the algebraic K-theory of an $\bbE_n$-semiring $\infty$-category is an $\bbE_n$-ring spectrum.
Finally, in Appendix \ref{comonoids} we show a relation of functors with comonoids, and in Appendix \ref{sec:app} we consider monoid, group, and ring objects from the perspective of Lawvere algebraic theories. \\

\noindent
\textbf{Conventions.} We freely use the language of $\infty$-categories throughout this paper.
In particular, we adopt the notational conventions of \cite{HTT} and \cite{HA} and provide more specific references where necessary.\\

\noindent
\textbf{Acknowledgements.}
We would like to thank Ulrich Bunke for suggesting that we work out these results in the setting of $\infty$-categories and for carefully reading a previous draft.
We'd also like to  thank Peter May, Tony Elmendorf and Mike Mandell for helpful comments and discussions. 

\section{\texorpdfstring{$\infty$}{Infinity}-categories of commutative monoids and groups}\label{sec:monoids}

Given an $\infty$-category $\calc$ with finite products, we may form the $\infty$-category $\Mon(\calc)$ of $\Eoo$-monoids in $\calc$.
By definition, an $\Eoo$-monoid $M$ in $\calc$ is a functor $M\colon\NFin\to\calc$ such that the morphisms $M(\n)\to M(\one)$ induced by the inert maps $\rho^i\colon\n\to\one$ exhibit~$M(\n)$ as an $n$-fold power of~$M(\one)$ in~$\calc$ (see \cite[2.1.1.8, 2.4.2.1, 2.4.2.2]{HA} for details).
In the terminology of~\cite{segal_categories},~$M$ is called a \emph{special}~$\Gamma$-object of~$\calc$.
In what follows we will sometimes abuse notation and also use the same name for the underlying object of such an~$\Eoo$-monoid.
Given an $\Eoo$-monoid~$M$, we obtain a (coherently associative and commutative) multiplication map 
\[m\colon M\times M\to M\, ,\]
uniquely determined up to a contractible space of choices.

\begin{prop}\label{prop:gpl}
Let $\calc$ be an $\infty$-category with finite products and let $M$ be an $\bbE_\infty$-monoid in~$\calc$. Then the following conditions are equivalent:
\begin{enumerate}
\item The $\Eoo$-monoid~$M$ admits an \emph{inversion} map, i.e., there is a map $i\colon M\to M$ such that the composition
\begin{equation*}
M \xrightarrow{\Delta} M \times M \xrightarrow{id \times i} M \times M \xrightarrow{m}  M 
\end{equation*}
is homotopic to the identity.
\item The commutative monoid object of $\Ho(\calc)$ underlying the $\Eoo$-monoid~$M$ is a \emph{group} object.
\item 
The \emph{shear map} $s\colon M \times M \to M \times M$, defined as the projection $pr_1\colon M \times M \to M$ on the first factor and the multiplication $m\colon M \times M \to M$ on the second factor, is an equivalence.
\item
The special~$\Gamma$-object $M\colon\NFin\to\calc$ is \emph{very special} (again in the terminology of~\cite{segal_categories}).
\end{enumerate}
\end{prop}
\begin{proof}
This follows immediately from the fact that $\calc\to \mathrm{N}(\Ho(\calc))$ is conservative and preserves products.
\end{proof}

\begin{defn}
Let~$\calc$ be an $\infty$-category with finite products. An object~$M\in\Mon(\calc)$ is called an {\em $\bbE_\infty$-group} in $\calc$ if it satisfies the equivalent conditions of \autoref{prop:gpl}. 
We write $\Grp(\calc)$ for the full subcategory of~$\Mon(\calc)$ consisting of the $\Eoo$-groups.
\end{defn}

\begin{rmk}
There are similar equivalent characterizations as in the proposition for $\bbE_n$-monoids,~$n\geq 1$. In fact, they can be applied more generally to algebras for monochromatic $\infty$-operads~$\calo$ equipped with a morphism $\bbE_1\to\calo$. In this case, these characterizations serve as a definition of $\calo$-\emph{groups}. 
Since an ordinary monoid having right-inverses is a group, we can use the fact that every morphism in~$\Ho(\calc)$ lifts to a morphism in~$\calc$ to conclude that also the characterizations (i) and (iii) are equivalent to their respective `two-sided variants', but in characterization (iv) one must instead use (very) special simplicial objects in~$\calc$.
\end{rmk}

\begin{rmk}
Recall (cf.\ \cite[Remark 5.1.3.3]{HA}) that an $\bbE_n$-monoid object $M$ of an $\infty$-topos $\calc$ is said to be {\em grouplike} if (the sheaf) $\pi_0 M$ is a group object.
In more general situations, such as for instance $\calc=\Cat_\infty$, the correct $\pi_0$ is unclear, and in any case the resulting notion of `grouplike monoid' may not agree with that of `group'.
\end{rmk}

\begin{rmk}
In our definition of a group object we force the `inversion' morphism to be an actual morphism of the underlying objects in $\calc$. In many situations, however, there is a natural inversion which is naturally only an anti-morphism. For example, this is the case in a tensor category with tensor inverses or in the category of Poisson Lie groups. This suggests that there should be a notion of group object with such an anti-inversion morphism. It would be interesting to study such a notion, though we will not need this.
\end{rmk}

Given two $\infty$-categories~$\calc$ and~$\cald$ with finite products, we write $\Fun^\Pi(\calc,\cald)$ for the $\infty$-category of finite product preserving functors from~$\calc$ to~$\cald$. If $\calc$ and $\cald$ are complete, we write $\FunR(\calc,\cald)$ for the $\infty$-category of limit preserving functors. In this situation, the $\infty$-category $\FunR(\calc,\cald)$ is also complete and limits in $\FunR(\calc,\cald)$ are formed pointwise in~$\cald$. This follows from the corresponding statement for $\Fun(\calc,\cald)$ and from the fact that such a pointwise limit of functors is again limit preserving.

\begin{lem}\label{lem:algR}
If $\calc$ and $\cald$ are $\infty$-categories with finite products, then $\Fun^\Pi(\calc,\cald)$ also has finite products and we have canonical equivalences
\[
\Mon\big(\Fun^\Pi(\calc,\cald)\big) \simeq \Fun^\Pi\big(\calc, \Mon(\cald) \big)
\]
and
\[
\Grp\big(\Fun^\Pi(\calc,\cald)\big) \simeq \Fun^\Pi\big(\calc, \Grp(\cald) \big).
\]
If $\calc$ and $\cald$ are complete, then so is $\FunR(\calc,\cald)$, and we have canonical equivalences
\[
\Mon\big(\FunR(\calc,\cald)\big) \simeq \FunR\big(\calc, \Mon(\cald) \big)
\]
and
\[
\Grp\big(\FunR(\calc,\cald)\big) \simeq \FunR\big(\calc, \Grp(\cald) \big).
\]
\end{lem}
\begin{proof}
We only give the proof of the second case, as the first one is entirely analogous.
As recalled above, an $\Eoo$-monoid in an $\infty$-category~$\cale$ is given by a functor $M\colon \NFin \to \cale$ satisfying the usual Segal condition, i.e., the inert maps $\n\to\one$ exhibit $M(\n)$ as the~$n$-fold power of~$M(\one)$. We denote the full subcategory spanned by such functors by 
\begin{equation*}
\Fun^{\times}\big(\NFin, \cale\big) \subseteq \Fun\big(\NFin, \cale\big).
\end{equation*}
Using this notation, we obtain a fully faithful inclusion
\begin{eqnarray*}
\Mon\big(\FunR(\calc,\cald)\big) &\simeq& \Fun^{\times}\big(\NFin,\FunR(\calc,\cald)\big)  \\
&\subseteq& \Fun\big(\NFin,\Fun(\calc,\cald)\big) \\
&\simeq& \Fun\big(\NFin \times \calc, \cald\big)
\end{eqnarray*}
whose essential image consists of those functors~$F$
such that $F(-,C)\colon\NFin\to\cald$ is special for all $C\in\calc$ and such that $F(\n,-)\colon\calc\to\cald$ preserves limits for all~$\n\in\NFin$. This follows from the fact that limits in $\FunR(\calc,\cald)$ are formed pointwise, as remarked above. 
In a similar vein, we obtain a fully faithful inclusion
\begin{eqnarray*}
\FunR\big(\calc, \Mon(\cald)\big) & \simeq& \FunR\big(\calc, \Fun^\times(\NFin, \cald)\big) \\
&\subseteq& \Fun\big(\calc, \Fun(\NFin, \cald)\big) \\
&\simeq& \Fun(\calc \times \NFin, \cald) \\
& \simeq& \Fun(\NFin \times \calc , \cald)
\end{eqnarray*}
with the same essential image, concluding the proof for the case of monoids. The proof for the case of groups works exactly the same. In fact, using characterization~(4) of \autoref{prop:gpl}, it suffices to replace special $\Gamma$-objects by very special $\Gamma$-objects. 
\end{proof}

\section{Preadditive and additive \texorpdfstring{$\infty$}{infinity}-categories}\label{sec:pre}

An $\infty$-category is preadditive if finite coproducts and products exist and are equivalent. More precisely, we have the following definition.

\begin{defn}
An $\infty$-category $\calc$ is \emph{preadditive} if it is pointed, admits finite coproducts and finite products, and the canonical morphism $C_1 \sqcup C_2 \to C_1 \times C_2$ is an equivalence for all objects $C_1, C_2 \in \calc$.
In this case any such object will be denoted by $C_1 \oplus C_2$ and will be referred to as a \emph{biproduct} of~$C_1$ and~$C_2$. 
\end{defn}
Let us collect a few immediate examples and closure properties of preadditive $\infty$-categories.

\begin{eg}
An ordinary category~$\calc$ is preadditive if and only if $\N(\calc)$ is a preadditive $\infty$-category. Products and opposites of preadditive $\infty$-categories are preadditive. Clearly any $\infty$-category equivalent to a preadditive one is again preadditive. Finally, if~$\calc$ is a preadditive~$\infty$-category and~$K$ is any simplicial set, then~$\Fun(K,\calc)$ is preadditive.
This follows immediately from the fact that (co)limits in functor categories are calculated pointwise (\cite[Corollary~5.1.2.3]{HTT}).
\end{eg}
\noindent
We will obtain more examples of preaddtive $\infty$-categories from the following proposition, which gives a connection to \S\ref{sec:monoids}.

\begin{prop}\label{pro:pre}
Let~$\calc$ be an $\infty$-category with finite coproducts and products. Then the following are equivalent:
\begin{enumerate}
\item\label{p1}
The $\infty$-category $\calc$ is preadditive.
\item\label{p2}
The homotopy category $\Ho(\calc)$ is preadditive.
\item\label{p3}
The $\infty$-operad $\calc^{\sqcup} \to \NFin$  as constructed in \cite[Construction~2.4.3.1]{HA} is cartesian (\cite[Definition~2.4.0.1]{HA}).
\item\label{p4}
The forgetful functor $\Mon(\calc) \to \calc$ is an equivalence.
\end{enumerate}
Moreover, $\Mon(\calc)$ is preadditive if $\calc$ has finite products.
\end{prop}
\begin{proof}
Let us begin by proving that the first two statements are equivalent. The direction \ref{p1}$\Rightarrow $\ref{p2} follows from the fact that the functor $\gamma\colon\calc\to \mathrm{N}(\Ho(\calc))$ preserves finite (co)products. For the converse direction, let us recall that a morphism in~$\calc$ is an equivalence if and only if~$\gamma$ sends it to an isomorphism. Now, by our assumption on $\Ho(\calc)$, the canonical map $C_1\sqcup C_2\to C_1\times C_2$ in~$\calc$ is mapped to an isomorphism under~$\gamma$ and is hence an equivalence. 

To show \ref{p1} $\Rightarrow$ \ref{p3} we only need to check that the symmetric monoidal structure $\calc^{\sqcup} \to \NFin$ exhibits finite tensor products (in this case the disjoint union) as products. But this follows directly from \ref{p1}.

Now assume \ref{p3} holds. Then by \cite[Corollary~2.4.1.8]{HA} there exists an equivalence of symmetric monoidal structures $\calc^\sqcup \simeq \calc^{\times}$. Thus we get an induced equivalence
\begin{equation*}
\Mon(\calc) \simeq \Alg_{\bbE_\infty}(\calc^{\times})  \simeq \Alg_{\bbE_\infty}(\calc^{\sqcup}) .
\end{equation*}
compatible with the forgetful functors to~$\calc$. But for the latter symmetric monoidal structure the forgetful functor $\Alg_{\bbE_\infty}(\calc^{\sqcup}) \to \calc$ always induces an equivalence, as shown in \cite[Corollary~2.4.3.10]{HA}.

Finally, assume \ref{p4} holds. Then in order to show that $\calc$ is preadditive it suffices to show that $\Mon(\calc)$ is preadditive. To see that $\Mon(\calc)$ is preadditive we note that 
limits in $\Mon(\calc)$ are formed as the limits of the underlying objects of $\calc$. In particular, the underlying object of the product in $\Mon(\calc)$ is given by the product of the underlying objects.
Coproducts are more complicated, but it is shown in \cite[Proposition~3.2.4.7]{HA} that the underlying object of the coproduct is formed by the tensor product of the underlying objects, i.e., by the product 
of the underlying objects in our case. Thus, the underlying object of the coproduct and the product are equivalent. But, by assumption, $\Mon(\calc)\to\calc$ is fully-faithful, so that we already have such an equivalence in~$\Mon(\calc)$. This implies~\ref{p1} and concludes the proof.
\end{proof}

\begin{cor}\label{cor:pipre}
Let~$\calc$ and $\cald$ be $\infty$-categories with finite products and suppose that either $\calc$ or $\cald$ is preadditive. Then the $\infty$-category $\Fun^\Pi(\calc,\cald)$ is preadditive.
\end{cor}
\begin{proof}
If $\cald$ is preadditive, then~$\Fun(\calc,\cald)$ is also preadditive, and clearly $\Fun^\Pi(\calc,\cald)\subseteq \Fun(\calc,\cald)$ is stable under products. In particular, given two product preserving functors $f,g:\calc\to\cald$, the pointwise product $f \times  g\colon \calc\to\cald$ again lies in~$\Fun^\Pi(\calc,\cald)$. Since (co)limits in~$\Fun(\calc,\cald)$ are calculated pointwise (\cite[Corollary~5.1.2.3]{HTT}), we can use the preadditivity of~$\cald$ to conclude that $f\times g$ is also the coproduct~$f\sqcup g$ of~$f$ and~$g$ in~$\Fun(\calc,\cald)$, and hence, a posteriori, also the coproduct in~$\Fun^\Pi(\calc,\cald)$. A similar reasoning yields a zero object in~$\Fun^\Pi(\calc,\cald)$, and we conclude that $\Fun^\Pi(\calc,\cald)$ is preadditive. 

The case in which $\calc$ is preadditive is slightly more involved. Recall that a product preserving functor $f\colon\calc\to\cald$ induces a functor $\Mon(\calc)\to\Mon(\cald)$ (simply by composing a special $\Gamma$-object in~$\calc$ with~$f$). Since products in $\infty$-categories of $\Eoo$-monoids are calculated in the underlying $\infty$-categories, this induced functor preserves products. Thus, we obtain a functor
\[
\Fun^\Pi(\calc,\cald)\to\Fun^\Pi(\Mon(\calc),\Mon(\cald)).
\]
By \autoref{pro:pre} we know that $\Mon(\cald)$ is preadditive. The first part of this proof implies the same for $\Fun^\Pi(\Mon(\calc),\Mon(\cald))$, and hence we are done if we can show that the above functor is an equivalence. A functor in the reverse direction is given by composition with the equivalence $\calc\simeq\Mon(\calc)$ (use \autoref{pro:pre} again) and with $\Mon(\cald)\to\cald.$ It is easy to check that the resulting endofunctor of $\Fun^\Pi(\calc,\cald)$ is equivalent to the identity, as is also the case for the other composition.
\end{proof}

\begin{cor}\label{cor:colmon}
Let $\calc$ be an $\infty$-category with finite products and let $\cald$ be a preadditive $\infty$-category. 
\begin{enumerate}
\item The $\infty$-category $\Mon(\calc)$ is preadditive.
\item The forgetful functor $\Mon(\Mon(\calc)) \to \Mon(\calc)$ is an equivalence.
\item There is an equivalence $\Fun^\Pi(\cald, \Mon(\calc)) \simeq \Fun^\Pi(\cald,\calc).$
\end{enumerate}
\end{cor}
\begin{proof}
The first assertion is a consequence of the proof of \autoref{pro:pre}. The second follows immediately from that same proposition, while the last statement is implied by \autoref{lem:algR} and the observation that $\Fun^\Pi(\cald,\calc)$ is preadditive whenever~$\cald$ is as guaranteed by \autoref{cor:pipre}.
\end{proof}

We now establish basically the analogous results for additive $\infty$-categories. As it is very similar to the case of preadditive $\infty$-categories, we leave out some of the details. Parallel to ordinary category theory, we introduce additive $\infty$-categories by imposing an additional exactness condition on preadditive~$\infty$-categories. Let~$\calc$ be a preadditive $\infty$-category and let~$A$ be an object of~$\calc$. We know from \autoref{pro:pre} that~$A$ can be canonically endowed with the structure of an $\Eoo$-monoid, and it is shown in \cite[Section~2.4.3]{HA} that this structure is given by the fold map $\nabla\colon A\oplus A\to A$. The {\em shear map}
\[s\colon A\oplus A\to A\oplus A\]
is the projection $pr_1\colon A\oplus A\to A$ on the first factor and the fold map $\nabla\colon A\oplus A\to A$ on the second. 

\begin{defn}
A preadditive $\infty$-category~$\calc$ is \emph{additive} if, for every object~$A\in\calc$, the shear map~$s\colon A\oplus A\stackrel{\sim}{\to} A\oplus A$ is an equivalence.
\end{defn}

\begin{egs}\label{egs:add}
An ordinary category~$\calc$ is additive if and only if $\N(\calc)$ is an additive $\infty$-category. Products and opposites of additive $\infty$-categories are additive. If $\calc$ is an additive $\infty$-category, then any $\infty$-category equivalent to $\calc$ is additive $\infty$-category, and any functor $\infty$-category $\Fun(K,\calc)$ is additive.
\end{egs}

The connection to $\Eoo$-groups and hence to \S\ref{sec:monoids} is provided by the following analog of \autoref{pro:pre}.

\begin{prop}\label{pro:add}
For an $\infty$-category~$\calc$ with finite products and coproducts, the following are equivalent:
\begin{enumerate}
\item\label{pp1} The $\infty$-category $\calc$ is additive.
\item\label{pp2} The homotopy category $\Ho(\calc)$ is additive.
\item\label{pp3} The forgetful functor $\Alggp(\calc) \to \calc$ is an equivalence.
\end{enumerate}
Moreover, if $\calc$ is an $\infty$-category with finite products, then $\Alggp(\calc)$ is additive.
\end{prop}
\begin{proof}
The proof of the equivalence of~\ref{p1} and~\ref{pp3} is parallel to the proof of \autoref{pro:pre}. To see that~\ref{pp1} implies~\ref{pp3} we note that by \autoref{pro:pre} 
we have an equivalence $\Alg_{\bbE_\infty}(\calc^{\sqcup}) \simeq\Mon(\calc)\to\calc$. But it is shown in \cite[Section~2.4.3]{HA} that an inverse to this equivalence endows an object~$A\in\calc$ with 
the algebra structure given by the fold map $\nabla\colon A\oplus A\to A$. Now, the statement that such an algebra object is grouplike is equivalent to the shear map being an equivalence.
 Thus, invoking~\ref{pp1}, we obtain an equivalence $\Mon(\calc)\simeq\Alggp(\calc)$, which gives~\ref{pp3}. Conversely, to see that~\ref{pp3} implies~\ref{pp1}, we need to show that $\Grp(\calc)$ is additive. Preadditivity is clear and additivity follows from the characterization of groups given in~ \autoref{prop:gpl}.
\end{proof}

\begin{cor}\label{cor:piadd}
Let~$\calc$ and $\cald$ be $\infty$-categories with finite products and suppose that either $\calc$ or $\cald$ is additive. Then the $\infty$-category $\Fun^\Pi(\calc,\cald)$ is additive.
\end{cor}

\begin{cor}\label{cor:colgrp}
Let $\calc$ be an $\infty$-category with finite products and let $\cald$ be an additive $\infty$-category. 
\begin{enumerate}
\item The $\infty$-category $\Grp(\calc)$ is additive.
\item The forgetful functor $\Grp(\Grp(\calc)) \to \Grp(\calc)$ is an equivalence.
\item There is an equivalence $\Fun^\Pi(\cald, \Grp(\calc)) \simeq \Fun^\Pi(\cald,\calc).$
\end{enumerate}
\end{cor}

\begin{rmk}
\autoref{cor:colmon} and \autoref{cor:colgrp} basically state that $\Mon(-)$ and $\Alggp(-)$ are colocalizations of the $\infty$-category of $\infty$-categories with finite products and product preserving functors. Much of the remainder of the paper makes use of this observation, although we prefer to phrase things slightly differently: namely, $\Mon(-)$ and $\Grp(-)$ also induce colocalizations of~$\Prr$, which in turn (using the anti-equivalence between~$\Prl$ and $\Prr$) induce localizations of~$\Prl$. We have opted to state our results in term of localizations as we think they are slightly more intuitive from this perspective.
\end{rmk}

\section{Smashing localizations}\label{sec:smash}

So far we have discussed $\infty$-categories with finite products. We now turn our attention to presentable $\infty$-categories. The primary purpose of this section is to review the notion of \emph{smashing localizations}, which we then specialize to $\Prlo$ in order to deduce some important consequences which will play an essential role throughout the remainder of the paper.

Let $\calc$ be an $\infty$-category.
Recall that a {\em localization} of $\calc$ is functor $L\colon\calc\to\cald$ which admits a fully faithful right adjoint $R\colon\cald\to\calc$.
If $L\colon\calc\to\cald$ is a localization, then $\cald$ is equivalent (via the fully faithful right adjoint) to a full subcategory $L\calc$ of $\calc$, called the subcategory of {\em local objects}.
For this reason we typically identify localizations with reflective subcategories (i.e., full subcategories such that the inclusion admits a left adjoint). We will also sometimes write $L$ for the endofunctor of $\calc$ obtained as the composite of $L\colon\calc\to\cald$ followed by the fully faithful right adjoint $R\colon\cald\to\calc$. Given such a localization, a map $X\to Y$ is a \emph{local equivalence} if $LX\to LY$ is an equivalence.

\begin{lem}
Let $\calc$ be an $\infty$-category and $M\colon\calc\to\calc$ an endofunctor equipped with a natural transformation $\eta\colon\id\to M$.
Then $M$ is equivalent to the composite $R\circ L$ of a localization $L\colon\calc\to\cald$ if and only if, for every object $X$ of $\calc$, the two obvious maps $M(X)\to M(M(X))$ are equivalences.
\end{lem}

\begin{proof}
This is condition (3) of \cite[Proposition 5.2.7.4]{HTT}.
\end{proof}

If $\calc$ has a symmetric monoidal structure $\calc^\otimes$, then it is sometimes the case that a localization of $\calc$ is given by `smashing' with a fixed object $I$ of $\calc$. In keeping with the terminology used in stable homotopy theory, we make the following definition.

\begin{defn}
Let $\calc^\otimes$ be a symmetric monoidal $\infty$-category. We  say that a localization $L\colon\calc\to\calc$ is {\em smashing} if it is of the form $L\simeq (-)\otimes I$ for some object~$I$ of $\calc$.
\end{defn}

Recall from \cite[Definition~6.3.2.1]{HA} that an \emph{idempotent object} in~$\calc^\otimes$ is an object~$I$ together with a morphism from the tensor unit such that the two obvious maps $I\to I\otimes I$ are equivalences. It follows that the endofunctor of $\calc$ given by tensoring with $I$ is a localization \cite[Proposition 6.3.2.4]{HA}.
Conversely for a smashing localization $L\simeq (-)\otimes I$ the object $I$ is necessarily an idempotent commutative algebra object of $\calc$.
In other words, showing that the functor $(-)\otimes I$ is a localization is the same as endowing $I$ with the structure of an idempotent commutative algebra object of $\calc$. This provides a one-to-one correspondence between smashing localizations and idempotent commutative algebra objects.

There are two obvious key features of smashing localizations: first, they preserve colimits (provided the tensor structure is compatible with colimits, which is always the case if it is closed), and second, they are symmetric monoidal in the sense of the following definition.

\begin{defn}
Let $\calc^\otimes$ be a symmetric monoidal $\infty$-category equipped with a localization $L\colon\calc\to\cald$ of the underlying $\infty$-category $\calc$. Then $L$ is {\em compatible with the symmetric monoidal structure} (or simply \emph{symmetric monoidal}) if, whenever $X\to Y$ is a local equivalence, then so is $X\otimes Z\to Y\otimes Z$  for any object $Z$ of $\calc$.
\end{defn}

Given such a localization, the subcategory $\cald\simeq L\calc$ of local objects inherits a symmetric monoidal structure from that of $\calc$. This is the content of the following lemma which also justifies the terminology \emph{symmetric monoidal} localization.
Identifying $\cald$ with the full subcategory $L\calc$ of local objects, let $R^\otimes\colon\cald^\otimes\subseteq\calc^\otimes$ be the inclusion of the full subcategory consisting of those objects $X_1\oplus\cdots\oplus X_n$ such that each $X_i$ is in $\cald$.

\begin{lem}\label{prop:symmonloc}
Let $\calc^\otimes$ be a symmetric monoidal $\infty$-category equipped with a symmetric monoidal localization $L\colon\calc\to\cald.$ Then there is a symmetric monoidal structure~$\cald^\otimes$ on $\cald$ such that $L$ extends to a symmetric monoidal functor $L^\otimes\colon\calc^\otimes\to\cald^\otimes$ and such that the right adjoint $R^\otimes\colon\cald^\otimes\to\calc^\otimes$ is lax symmetric monoidal.
\end{lem}
\begin{proof}
This is a special case of \cite[Proposition 2.2.1.9]{HA}.
\end{proof}

\begin{rmk}\label{rmk_closed}
If $\calc^\otimes$ is a closed symmetric monoidal $\infty$-category equipped with a symmetric monoidal localization $L\colon \calc \to \calc$. Then $L$ is compatible with the closed structure in the sense that, for every pair of objects $C$ and $D$ of $\calc$, the localization $C\to LC$  induces an equivalence
\[
D^{LC} \simeq D^C
\]
whenever $D$ is local. This follows immediately from the definition.
\end{rmk}

\begin{lem}\label{prop:algloc}
Let $\calc^\otimes$ be a symmetric monoidal $\infty$-category equipped with a symmetric monoidal localization $L\colon\calc\to\cald,$ and let $R\colon\cald\to\calc$ denote the right adjoint of $L$.
Then there is an induced localization $L'\colon\Alg_{\bbE_\infty}(\calc)\to\Alg_{\bbE_\infty}(\cald)$ such that the diagram
\[
\xymatrix{
\Alg_{\bbE_\infty}(\calc)\ar[r]^{L'}\ar[d] & \Alg_{\bbE_\infty}(\cald)\ar[d]\\
\calc\ar[r]^L & \cald}
\]
commutes. Moreover, given $A\in\Alg_{\Eoo}(\calc)$, there exists a unique commutative algebra structure on $RLA$ such that unit map $A\to RLA$ extends to a morphism of commutative algebras.
\end{lem}
\begin{proof}
By \autoref{prop:symmonloc} above, we obtain maps $L'\colon\Alg_{\bbE_\infty}(\calc)\to\Alg_{\bbE_\infty}(\cald)$ and $R'\colon\Alg_{\bbE_\infty}(\cald)\to\Alg_{\bbE_\infty}(\calc)$ by composing sections $\bbE_\infty\to\calc^\otimes$ with $L^\otimes$ and sections $\bbE_\infty\to\cald^\otimes$ with $R^\otimes$, respectively. In a similar fashion we also obtain unit and counit transformations such that the counit is an equivalence. It follows that $L'$ is a localization. 

For the second assertion, we know already that $R'L'A$ comes with a canonical commutative algebra map $\eta'\colon A\to R'L' A$, the adjunction unit evaluated at $A$, and that this map extends the adjunction unit $\eta\colon A\to RLA$ of the underlying objects. If $\eta''\colon A\to R'B$ is a second such map of commutative algebras, then the universality of~$\eta'$ implies that $\eta''$ factors essentially uniquely as 
\[
\phi\circ\eta'\colon A\to R'L'A\to R'B.
\]
Since the underlying map of $\phi$ is an identity, if follows that $\phi$ itself is an equivalence since $\Alg_{\Eoo}(\calc)\to\calc$ is conservative. We can now conclude since the space of reflections of a fixed object in a full subcategory is contractible if non-empty.
\end{proof}

\begin{rmk}
The second part of the lemma implies that $RLA$ can be turned into an $\Eoo$-algebra such that the unit map $A\to RLA$ can be enhanced to a morphism of $\Eoo$-algebras. Moreover, the space of such enhancements is contractible. In particular, if $RLA$ is endowed with two different $\Eoo$-algebra structures, then the identity morphism of the underlying objects in~$\cald$ can be essentially uniquely turned into an equivalence of these two $\Eoo$-algebras compatible with the localizations. We will apply this in \S\ref{sec:spec} to smashing localizations on~$\Prl.$ 
\end{rmk}

\medskip
Now we specialize to the case of the (very large) $\infty$-category $\Prl$ of presentable $\infty$-categories and colimit-preserving functors. We will write $\calc$, $\cald$, etc.\ for objects of $\Prl$. Recall that $\Prl$ admits a closed symmetric monoidal structure which is uniquely characterized as follows: given presentable $\infty$-categories $\calc$ and~$\cald$, their tensor product $\calc\otimes\cald$ corepresents the functor $\Prl\to\widehat{\Cat}_\infty$ which sends $\cale$ to 
\[
\FunLL(\calc\times\cald,\cale)\subseteq\Fun(\calc\times\cald,\cale),
\]
the full subcategory consisting of those functors $F\colon\calc\times\cald\to\cale$ which preserve colimits separately in each variable. The unit of this monoidal structure on $\Prl$ is the $\infty$-category $\cals$ of spaces, as follows from the fact that $\FunL(\cals,\calc)\simeq\calc$ (\cite[Example~6.3.1.19]{HA}). Moreover, by \cite[Proposition~6.3.1.16]{HA} this tensor product admits the following description
\[
\calc\otimes\cald\simeq\FunR(\calc\op,\cald).
\]
Recall that $\FunL(\calc,\cald)$ is presentable (\cite[Propositon~5.5.3.8]{HTT}). It is immediate from the definition of $\calc\otimes\cald$ as a corepresenting object that the symmetric monoidal structure on $\Prl$ is closed, with right adjoint to $\calc\otimes(-):\Prl\to\Prl$ given by $\FunL(\calc,-):\Prl\to\Prl$. Lastly, the (possibly large) mapping spaces in $\Prl$ are given by the formula
\[
\Map_{\Prl}(\calc,\cald)\simeq\FunL(\calc,\cald)^\sim,
\]
the maximal subgroupoid. This description will be applied in \S\ref{sec:mon} to our context of monoids and groups.

\begin{prop}\label{prop:smmod}
Let $L\colon \Prl \to \Prl$ be a smashing localization or, more generally, a symmetric monoidal localization, and let $\calc$ and $\cald$ be presentable $\infty$-categories such that $\cald$ is in the essential image of $L$.
\begin{enumerate}
\item The map $\FunL(L\calc,\cald) \to \FunL(\calc,\cald)$ induced by the localization $\calc\to L\calc$ is an equivalence.  
\item If $L$ is smashing, then, writing $L\Prl$ for the image of $L$, there is an equivalence of $\infty$-categories
\[
L\Prl\simeq \Mod_{L\cals}(\Prl).
\]
\item Given a second symmetric monoidal localization $L'\colon\Prl\to\Prl$ such that $L'\Prl\subseteq L\Prl$, then the canonical morphism $L\calc \to L'\calc$ induces an equivalence $\FunL(L'\calc,\cald) \to \FunL(L\calc,\cald)$ for every $L'$-local~$\cald.$ 
\end{enumerate}
\end{prop}
\begin{proof}
The first statement follows from~\autoref{rmk_closed} and the second from \cite[Proposition 6.3.2.10]{HA}. Finally, the third one follows immediately from the first and the 2-out-of-3 property of equivalences.
\end{proof}

Let us now consider a presentable $\infty$-category endowed with a closed symmetric monoidal structure~$\calc^\otimes$. In this  context the closedness is equivalent to the fact that the monoidal structure preserves colimits separately in each variable, i.e., $\calc^\otimes$ is essentially just a commutative algebra object in $\Prl$~(\cite[Remark~6.3.1.9]{HA}). 

\begin{prop}\label{prop:smmonoidal}
Let $L\colon \Prl \to \Prl$ be a smashing localization or, more generally, a symmetric monoidal localization. Let $\calc^\otimes$ and $\cald^\otimes$ be closed symmetric monoidal presentable $\infty$-categories.
\begin{enumerate}
\item The $\infty$-category $L\calc$ admits a {\em unique} closed symmetric monoidal structure such that the localization map $\calc \to L\calc$ is a symmetric monoidal functor.
\item The map $\FunLMon(L\calc,\cald) \to \FunLMon(\calc,\cald)$ induced by the localization $\calc\to L\calc$ is an equivalence whenever $\cald$ is $L$-local.
\item Given a second symmetric monoidal localization $L'\colon\Prl\to\Prl$ such that $L'\Prl\subseteq L\Prl$, the induced morphism $L\calc \to L'\calc$ admits a unique symmetric monoidal structure. In particular, for every $L'$-local~$\cald$ the induced map $\FunLMon(L'\calc,\cald) \to \FunLMon(L\calc,\cald)$ is an equivalence.
\end{enumerate}
\end{prop}
\begin{proof}
Statement~(i) follows from~\autoref{prop:algloc}, which also gives equivalences
\[
\Map(\Delta^0,\FunLMon(L\calc,\cald^K))\simeq\Map(\Delta^0,\FunLMon(\calc,\cald^{K}))
\]
for any simplicial set $K$ such that $\cald^K$ is local.
(ii) then follows from the fact that $\Alg_{\bbE_\infty}(\Prl)$ is cotensored over $\Cat_\infty$ in such a way that $\cald^K$ is local whenever $\cald$ is local; indeed, the cotensor $\cald^K$ is given by the internal mapping object $\FunL(\calp(K),\cald)$, and this is a local object since $(-)\otimes\calp(K)$ preserves local equivalences by assumption.
Finally, (iii) is obtained by the same argument as (i) after replacing $\Prl$ with $L\Prl$, which has an induced closed symmetric monoidal structure, $L\colon\Prl\to\Prl$ with the functor $L\Prl\to L\Prl$ induced by the composite $\Prl\to L'\Prl\subseteq L\Prl$, and $\calc$ with $L\calc$, which also inherits a closed symmetric monoidal structure. 
\end{proof}

We shall see in the next section that formation of $\infty$-categories of commutative monoid and group objects in a presentable $\infty$-category $\calc$ are instances of smashing localizations of $\Prl$.
For the moment, it is worth mentioning that there are other well-known examples of smashing localizations of $\Prl$.
The most obvious one is the functor which associates to a presentable $\infty$-category $\calc$ its $\infty$-category $\calc_*$ of {\em pointed objects}; the fact that this is a smashing localization follows from the formula
\[
\calc_*\simeq\calc\otimes\cals_*
\]
and the fact that $\cals_*$ is an idempotent object of $\Prl$ (cf.~\cite[Proposition 6.3.2.11]{HA}).
An important feature of $\cals_*$ is that it is symmetric monoidal under the {\em smash product}, which is uniquely characterized by the requirement that the unit map $\cals\to\cals_*$ is symmetric monoidal.
Less obvious but possibly more important is the functor which associates to a presentable $\infty$-category $\calc$ the $\infty$-category $\Sp(\calc)$ of spectrum objects in $\calc$ (cf.\ \cite[Proposition 6.3.2.18]{HA}).

\section{Commutative monoids and groups as smashing localizations}\label{sec:mon}

In this section we show that the passage to $\infty$-categories of commutative monoids or groups are instances of smashing localizations of~$\Prl$.

\begin{prop}\label{prop:monpres}
Given a presentable $\infty$-category~$\calc$, then also the $\infty$-categories $\Mon(\calc)$ and $\Grp(\calc)$ are presentable.
\end{prop}
\begin{proof}
By definition the $\infty$-categories $\Mon(\calc)$ and $\Grp(\calc)$ are full subcategories of the presentable $\infty$-category $\Fun(\NFin,\calc).$ Therefore, it suffices to show that the monoids and groups, respectively, are precisely the $S$-local objects for a small collection $S$ of morphisms in $\Fun(\NFin,\calc)$ (\cite[Proposition~5.5.4.15]{HTT}). We will give the details for the case of monoids and leave the case of groups to the reader. 

In order to define $S$ we first note that the evaluation functors 
\[
\ev_{\n}\colon\Fun(\NFin,\calc) \to \calc
\]
admit left adjoints $F_{\n}\colon \calc \to \Fun(\NFin,\calc).$ Now, $M \in \Fun(\NFin,\calc)$ belongs to $\Mon(\calc)$ if for every $n \in \mathbb{N}$ the morphism $M(\n) \to \prod M(\one)$ is an equivalence in $\calc$, and this is the case if and only if for every $C \in \calc$ the morphism
\begin{equation}\label{strmorphisms}
\Map_\calc\big(C, M(\n)\big) \longrightarrow \prod  \Map_\calc\big(C, M(\one)\big)
\end{equation}
is an equivalence of spaces. Since $\calc$ is accessible it suffices to check this for objects in~$\calc^\kappa$, the essentially small subcategory of $\kappa$-compact objects for some regular cardinal~$\kappa$. Now we use the equivalences 
\[
\Map_\calc\big(C, M(\n)\big) \simeq \Map_{\Fun(\NFin,\calc)}(F_{\n}(C), M)
\]
and
\[
\qquad\prod \Map_\calc\big(C, M(\one)\big) \simeq \Map_{\Fun(\NFin,\calc)}\big(\bigsqcup F_{\one}(C), M\big)
\] 
and see that the morphism~\eqref{strmorphisms} is induced by a morphism 
$
\phi_{n,C}\colon\bigsqcup_n F_{\one}(C) \to F_{\n}(C) 
$
in $\Fun(\NFin,\calc)$.
Thus we may take $S$ to consist of the $\phi_{n,C}$, where $C$ ranges over any small collections of objects of $\calc$ which contains a representative of each equivalence class of object in $\calc^\kappa$.
\end{proof}

\begin{rmk}
The proof for groups is similar, though we have to add more maps to the set $S$ to account for the `very' special condition.
This tells us in particular that $\Grp(\calc)$ is a reflective subcategory of $\Mon(\calc).$ 
\end{rmk}

\begin{cor}
Let $\calc$ be a presentable $\infty$-category. Then there are functors 
\[
\calc \to \Mon(\calc)\to\Grp(\calc)
\]
which are left adjoint to the respective forgetful functors.
\end{cor}
\begin{proof}
 Since limits in $\Mon(\calc)$ and $\Grp(\calc)$ are computed as the limits of the underlying objects, this follows from the adjoint functor theorem.
\end{proof}

\begin{rmk}
Let~$\calc$ be a presentable~$\infty$-category. The functor $\Mon(\calc)\to\Grp(\calc)$ left adjoint to the forgetful functor $\Grp(\calc)\to\Mon(\calc)$ is called the \emph{group completion}.
Thus, in the framework of $\infty$-categories, the group completion $\Mon(\calc)\to\Grp(\calc)$ has the expected universal property, defining a left adjoint to the forgetful functor $\Grp(\calc)\to\Mon(\calc)$.
\end{rmk}

The following theorem, while straightforward to prove, is central.

\begin{thm}\label{thm:idem}
The assignments $\calc \mapsto \Mon(\calc)$ and $\calc \mapsto \Grp(\calc)$ refine to smashing localizations of $\Prl$. Thus, we have, in particular, equivalences of $\infty$-categories
\begin{equation*}
\Mon(\calc) \simeq \calc \otimes \Mon(\cals) \qquad\text{and}\qquad \Grp(\calc) \simeq \calc \otimes \Grp(\cals). 
\end{equation*}
The local objects are precisely the preadditive presentable $\infty$-categories and the additive presentable $\infty$-categories, respectively.
\end{thm}
\begin{proof}
The  description of the tensor product of presentable $\infty$-categories together with~\autoref{lem:algR} gives us the chain of equivalences
\begin{eqnarray*}
\calc \otimes \Mon(\cald) &\simeq& \FunR\big(\calc\op, \Mon(\cald)\big) \\
&\simeq& \Mon\big(\FunR(\calc\op,\cald)\big) \\
&\simeq&  \Mon(\calc \otimes \cald).
\end{eqnarray*}
In particular, we have $\Mon(\calc) \simeq \calc \otimes \Mon(\cals)$. The fact that $\Mon$ is a localization follows from~\autoref{cor:colmon}. The local objects are precisely the presentable $\infty$-categories $\calc$ for which the canonical functor is an equivalence
$\Mon(\calc) \simeq \calc$, hence by~\autoref{pro:pre} precisely the preadditive $\infty$-categories.
The case of groups is established along the same lines.
\end{proof}

As a consequence we obtain the following result.

\begin{cor}\label{cormon}
Let $\calc$ and $\cald$ be presentable $\infty$-category. Then there are canonical equivalences
\begin{align*}
&\calc \otimes \Mon(\cald) \simeq \Mon(\calc \otimes \cald) \simeq \Mon(\calc) \otimes \cald,  \\
&\calc \otimes \Grp(\cald)\; \simeq \Grp(\calc \otimes \cald)\: \simeq \Grp(\calc) \otimes \cald.
\end{align*}
\
\end{cor}

Let us denote the full subcategories of $\Prl$ spanned by the preadditive and additive $\infty$-categories respectively by
\[
\Prlpre\subseteq\Prl \qquad \text{and} \qquad \Prladd\subseteq\Prl.
\] 
Then \autoref{prop:smmod} specializes to the following two corollaries.

\begin{cor}\label{cor:premod}
The forgetful functors
\[
\Mod_{\Mon(\cals)}(\Prl)\to\Prl\quad\text{and}\quad\Mod_{\Grp(\cals)}(\Prl)\to\Prl
\]
induce equivalences of $\infty$-categories
\[
\mathrm{Mod}_{\Mon(\cals)}(\Prl)\simeq\Prlpre \qquad \text{and} \qquad \Mod_{\Grp(\cals)}(\Prl)\simeq\Prladd.
\]
\end{cor}

\begin{cor}\label{cor:frepre}
Let $\calc$ and $\cald$ be presentable $\infty$-categories.
\begin{enumerate}
\item If $\cald$ is preadditive then the free $\Eoo$-monoid functor $\calc\to\Mon(\calc)$ induces an equivalence of $\infty$-categories
\[\FunL(\Mon(\calc),\cald)\stackrel{\simeq}{\to}\FunL(\calc,\cald),\]
exhibiting $\Mon(\calc)$ as the \emph{free preadditive presentable $\infty$-category generated by~$\calc$}. In particular, we have canonical equivalences
\[\FunL(\Mon(\cals),\cald)\stackrel{\simeq}{\to}\FunL(\cals,\cald)\stackrel{\simeq}{\to}\cald\]
exhibiting $\Mon(\cals)$ as the free preadditive presentable~$\infty$-category on one generator.
\item If $\cald$ is additive then the free $\Eoo$-group functor $\calc\to\Grp(\calc)$ induces an equivalence of $\infty$-categories
\[\FunL(\Grp(\calc),\cald)\stackrel{\simeq}{\to}\FunL(\calc,\cald),\]
exhibiting $\Grp(\calc)$ as the \emph{free additive presentable $\infty$-category generated by~$\calc$}. In particular, the free $\Eoo$-group functor $\cals\to\Grp(\cals)$ induces canonical equivalences
\[\FunL(\Grp(\cals),\cald)\stackrel{\simeq}{\to}\FunL(\cals,\cald)\stackrel{\simeq}{\to}\cald\]
exhibiting $\Grp(\cals)$ as the free additive, presentable~$\infty$-category on one generator.
\end{enumerate} 
\end{cor}

The results of this section give us a refined picture of the stabilization process of presentable $\infty$-categories as we describe it in the next corollary (we will obtain a further \emph{monoidal} refinement in~\autoref{cor:monstab}). In \cite[Chapter~1]{HA} it is shown that the stabilization of a presentable $\infty$-category~$\calc$ is given by the $\infty$-category $\Sp(\calc)$ of \emph{spectrum objects} in $\calc$, which is to say the limit
\[
\Sp(\calc)\simeq\lim\{\calc_*\overset{\Omega}{\longleftarrow}\calc_*\overset{\Omega}{\longleftarrow}\calc_*\overset{\Omega}{\longleftarrow}\cdots\}\, ,
\]
taken in the $\infty$-category of (not necessarily small) $\infty$-categories, or equivalently in the $\infty$-category $\Prr$ of presentable $\infty$-categories by \cite[Theorem 5.5.3.18]{HTT}.
Alternatively, $\Sp(\calc)$ is equivalent to the $\infty$-category of reduced excisive functors
\[
\Sp(\calc)\simeq\mathrm{Exc}_*(\cals_*^\mathrm{fin},\calc)
\]
(see \cite[Section 1.4.2]{HA} for details).
Recall from \cite[Proposition~1.4.4.4]{HA} that for such a~$\calc$ the $\infty$-category $\Sp(\calc)$ is related to $\calc$ by the suspension spectrum adjunction $(\Sigma^\infty_+,\Omega^\infty_-)\colon\calc\adj\Sp(\calc).$ 

\begin{cor}\label{cor:stab}
The stabilization of presentable $\infty$-categories $\Prl\to\Prlst$ factors as a composition of adjunctions
$$
\Prl\adj\Prlpt\adj\Prlpre\adj\Prladd\adj\Prlst.
$$
In particular, if $\calc$ is a presentable $\infty$-category, then $\Sigma^\infty_+\colon\calc\to\Sp(\calc)$ factors as a composition of left adjoints 
\begin{equation*}
\Sigma^\infty_+\colon\calc\to\calc_*\to \Mon(\calc) \to \Grp(\calc) \to \Sp(\calc),
\end{equation*}
each of which is uniquely determined by the fact that it commutes with the corresponding free functors from~$\calc$. 
\end{cor}
\begin{proof}
This follows from~\autoref{cor:frepre} and the corresponding corollary for the functor $(-)_+\colon\calc\to\calc_\ast$ together with the facts that $\Sp(\calc)$ is additive (by Corollary~1.4.2.17 and Remark~1.1.3.5 in~\cite{HA}), $\Grp(\calc)$ is preadditive (even additive by~\autoref{cor:colgrp}), and $\Mon(\calc)$ is pointed (in fact, preadditive by~\autoref{cor:colmon}). For the second statement, it suffices to use \autoref{prop:smmod}. 
\end{proof}

\section{Canonical symmetric monoidal structures}\label{sec:spec}

Let us now assume that $\calc$ is a presentable $\infty$-category endowed with a closed symmetric monoidal structure~$\calc^\otimes$. In this section we specialize the general results from \S\ref{sec:smash} (or more specifically
\autoref{prop:smmonoidal}) to the localizations 
$(-)_\ast, \Mon(-), \Grp(-),$ and $\Sp(-)$. The two cases of~$\calc_\ast$ and $\Sp(\calc)$ are already essentially covered in \cite[Section~6.1.9]{HA}, but since these results are not stated explicitly, we include them here for the sake of completeness.

\begin{thm}\label{thm:tensorproducts}
Let $\calc^\otimes$ be a closed symmetric monoidal structure on a presentable $\infty$-category~$\calc$.
The $\infty$-categories $\calc_*$, $\Mon(\calc)$, $\Grp(\calc)$, and $\Sp(\calc)$ all admit closed symmetric monoidal structures, which are uniquely determined by the requirement that the respective free functors from~$\calc$ are symmetric monoidal.
Moreover, each of the functors
\begin{equation*}
\calc_*\to\Mon(\calc) \to \Grp(\calc) \to \Sp(\calc)
\end{equation*}
uniquely extends to a symmetric monoidal functor.
\end{thm}
\begin{proof}
Follows directly from the fact that the localizations are smashing using \autoref{prop:smmonoidal}.
\end{proof}
\noindent
From now on, when considered as symmetric monoidal $\infty$-categories, these $\infty$-categories are always endowed with the canonical monoidal structures of the theorem.

\begin{warn}
The reader should not confuse the two symmetric monoidal structures on~$\calc$ that are used in the above construction. The first one is the cartesian structure~$\calc^\times$ which is used to define the $\infty$-category $\Mon(\calc)$ of $\Eoo$-monoids. The second one is the closed symmetric monoidal structure $\calc^\otimes$ which induces a monoidal structure on $\Mon(\calc)$ as described in the theorem. In applications, these two monoidal structures on~$\calc$ often agree, which amounts to assuming that~$\calc$ is cartesian closed. This is the case in the most important examples, namely $\infty$-topoi (such as $\cals$) and $\Cat_\infty$.  
\end{warn}

\begin{eg}
\begin{enumerate}
\item The (nerve of the) category $\Set$ of sets is a cartesian closed presentable $\infty$-category, and $\Grp(\Set)$ is just the (nerve of the) category $\Ab$ of abelian groups. The free functor $\Set\to\Ab$ can then of course be turned into a symmetric monoidal functor with respect to the cartesian product on $\Set$ and the usual tensor product on $\Ab$. Thus, in this very special case, the theorem reproduces the classical tensor product of abelian groups. 
\item The $\infty$-category~$\cals$ of spaces is a cartesian closed presentable $\infty$-category. The $\infty$-category $\Mon(\cals)$ of $\Eoo$-spaces hence comes with a canonical closed symmetric monoidal structure, as does the $\infty$-category~$\Grp(\cals)$ of grouplike $\Eoo$-spaces. Since the latter $\infty$-category is equivalent to the $\infty$-category of connective spectra (\cite[Remark~5.1.3.7]{HA}),  the canonical symmetric monoidal structure on $\Grp(\cals)$ agrees with the smash product of connective spectra.
\item Let $\Cat$ denote the cartesian closed presentable $\infty$-category of small ordinary categories (this is actually a 2-category, in the sense of \cite[Section 2.3.4]{HTT}). 
Thus, the $\infty$-category $\SymMonCat \simeq \Mon(\Cat)$ of small symmetric monoidal categories admits a canonical closed symmetric monoidal structure such that the free functor 
$\Cat \to \SymMonCat$ can be promoted to a symmetric monoidal functor in a unique way. 
This structure on $\SymMonCat$ has been explicitly constructed and discussed in the literature
(see \cite{Power} and the more explicit \cite{schmitt2007tensor}). In fact, this tensor product is slightly subtle since, at least to the knowledge of the authors, it can not be realized as a symmetric monoidal structure on the 1-category of small categories (as opposed to the $2$-category $\Cat$).
\item The $\infty$-category $\Catoo$ of small $\infty$-categories is a cartesian closed presentable $\infty$-category. Thus, as an $\infty$-categorical variant of the previous example, we obtain a canonical closed symmetric monoidal structure on the $\infty$-category $\SymMonCatoo$ of small symmetric monoidal $\infty$-categories.
\end{enumerate}
\end{eg}

We have already seen that, for presentable~$\infty$-categories~$\calc$, the passage to commutative monoids and commutative groups has a universal property (\autoref{cor:frepre}). In the case of closed symmetric monoidal presentable $\infty$-categories we now obtain a refined universal property for the symmetric monoidal structures of \autoref{thm:tensorproducts}. For convenience, we also collect the analogous results for the passage to pointed objects and spectrum objects. 

\begin{prop}\label{universalmonoidal}
Let $\calc$ and $\cald$ be closed symmetric monoidal presentable $\infty$-categories. 
\begin{enumerate}
\item If~$\cald$ is pointed then the symmetric monoidal functor $\calc \to \calc_\ast$ induces an equivalence of $\infty$-categories
\begin{equation*}
  \FunLMon(\calc_\ast, \cald) \to \FunLMon(\calc, \cald).
\end{equation*}
\item If~$\cald$ is preadditive then the symmetric monoidal functor $\calc \to \Mon(\calc)$ induces an equivalence of $\infty$-categories
\begin{equation*}
  \FunLMon(\Mon(\calc), \cald) \to \FunLMon(\calc, \cald).
\end{equation*}
\item If $\cald$ is additive then the symmetric monoidal functor $\calc \to \Grp(\calc)$ induces an equivalence of $\infty$-categories
\begin{equation*}
  \FunLMon(\Grp(\calc), \cald) \to \FunLMon(\calc, \cald).
\end{equation*}
\item If~$\cald$ is stable then the symmetric monoidal functor $\calc \to \Sp(\calc)$ induces an equivalence of $\infty$-categories
\begin{equation*}
  \FunLMon(\Sp(\calc), \cald) \to \FunLMon(\calc, \cald).
\end{equation*}
\end{enumerate}
\end{prop}
\begin{proof} 
This follows immediately from the second statement of \autoref{prop:smmonoidal}.
\end{proof}

Here is the monoidal refinement of the stabilization process which is now an immediate consequence of the third statement of \autoref{prop:smmonoidal}.

\begin{cor}\label{cor:monstab}
\begin{enumerate}
\item Let $\calc$ and $\cald$ be closed symmetric monoidal presentable $\infty$-categories and let us consider a symmetric monoidal left adjoint $F\colon\calc\to\cald.$ In the following commutative diagram, each of the functors induced by~$F$ admits a symmetric monoidal structure:
\[
\xymatrix{
\calc\ar[d]\ar[r]&\calc_\ast\ar[d]\ar[r]&\Mon(\calc)\ar[d]\ar[r]&\Grp(\calc)\ar[d]\ar[r]&
\Sp(\calc)\ar[d]\\
\cald\ar[r]&\cald_\ast\ar[r]&\Mon(\cald)\ar[r]&\Grp(\cald)\ar[r]&\Sp(\cald_\ast)
}
\]
Moreover, these symmetric monoidal structures are uniquely characterized by the fact that the functors commute with the free functors from~$\calc$.
\item The stabilization of presentable $\infty$-categories $\Prl\to\Prlst$ admits a symmetric monoidal refinement $\Prlo\to\Prlsto$ which factors as a composition of adjunctions
\[
\Prlo\adj\Prlpto\adj\Prlpreo\adj\Prladdo\adj\Prlsto.
\]
\end{enumerate}
\end{cor}

\begin{rmk}
\begin{enumerate}
\item One can use the theory of $\Gamma$-objects in $\calc$ to obtain a more concrete description of the tensor product on $\Mon(\calc)$ and $\Grp(\calc)$ as the convolution product (see \cite[Corollary~6.3.1.12]{HA} for the case in which $\calc$ is the $\infty$-category of spaces).
\item The uniqueness of the symmetric monoidal structures can be used to compare our results to existing ones. Every simplicial combinatorial, monoidal model category leads to a presentable, closed symmetric monoidal $\infty$-category. Thus for the monoidal model category of $\Gamma$-spaces as discussed in \cite{Schwede} it follows immediately that the symmetric monoidal structure on the underlying $\infty$-category has to agree with our structure. The same applies to the model structure on $\Gamma$-objects in any nice model category, for example in presheaves as discussed in~\cite{Bergsaker}. 
\end{enumerate}
\end{rmk}

\section{More functoriality}\label{sec:more}

In \S\ref{sec:mon} we saw that for presentable $\infty$-categories the passages to commutative monoids and groups are smashing localizations and hence, in particular, define functors 
$$\Mon(-),\Grp(-): \,\Prl\to\Prl\ .$$
But this passage allows for more functoriality. In fact, a product-preserving functor $F\colon \calc \to \cald$ induces functors
\[
\underline{F}\colon \Mon(\calc) \to \Mon(\cald)\qquad\text{and}\qquad
\underline{F}\colon \Grp(\calc) \to \Grp(\cald)
\]
simply by post-composing the respective (very) special $\Gamma$-objects with~$F$. The main goal of this section is to establish~\autoref{cor:lax}, which states that under certain mild assumptions these extensions themselves are lax symmetric monoidal with respect to the canonical symmetric monoidal structures established in~\autoref{thm:tensorproducts}. This corollary will be needed in our applications to algebraic K-theory in \S\ref{sec:infinite}.
We begin by comparing these two potentially different functorialities of the assignments~$\calc\mapsto\Mon(\calc)$ and $\calc\mapsto\Grp(\calc).$ 

\begin{lem}\label{leftright}
Let $L\colon\calc\rightarrow\cald$ be a functor of presentable $\infty$-categories with right adjoint $R\colon\cald\to\calc$.
\begin{enumerate}
\item
If $L\colon\calc\to\cald$ is product-preserving and if products in $\calc$ and $\cald$ commute with countable colimits, then the functors 
\begin{equation*}
\Mon(L)\colon \Mon(\calc) \to \Mon(\cald) \qquad\text{and}\qquad
\underline{L}\colon \Mon(\calc) \to \Mon(\cald) 
\end{equation*}
described above are equivalent.
\item
The canonical extension $\underline{R}\colon \Mon(\cald) \to \Mon(\calc)$ is right adjoint to the functor $\Mon(L)$.
\end{enumerate}
The corresponding two statements for $\Eoo$-groups hold as well.
\end{lem}
\begin{proof}
For the first claim we must show that if $L$ preserves products then the two functors agree. This follows if we can show that $\underline {L}$ is a left adjoint and the diagram 
\begin{equation}\label{dia:L}
\xymatrix{
\calc \ar[rr]^-L\ar[d]_{\Fr} && \cald \ar[d]^{\Fr}\\
\Mon(\calc) \ar[rr]_-{\underline{L}} && \Mon(\cald)
} 
\end{equation}
commutes in $\widehat{\Cat}_\infty$. To see that $\underline{L}$ is left adjoint we observe that it commutes with sifted colimits, as they are detected by the forgetful functors $\Mon(\calc)\to\calc$ and $\Mon(\cald)\to\cald$, and also that it commutes with coproducts, as coproducts in $\Mon(\calc)$ and $\Mon(\cald)$ are given by the tensor product which is preserved by~$L$. To conclude this part of the proof it suffices to show that there is an equivalence $\Fr\circ L\simeq\underline{L}\circ\Fr$. For this, we consider the mate of the equivalence $L\circ U\simeq U\circ\underline{L}\colon\Mon(\calc)\to\cald$, i.e., we form the following pasting with the respective adjunction morphisms:
\[
\xymatrix{
\calc\ar[r]^-{\Fr}\ar@/_0.8pc/[dr]_-=&\Mon(\calc)\ar[d]_-U\ar[r]^-{\underline{L}}&\Mon(\cald)\ar[d]^-U\ar@/^1.0pc/[rd]^-=&\\
&\calc\ar[r]_-L&\cald\ar[r]_-{\Fr}&\Mon(\cald)
}
\]
In order to show that the resulting transformation 
\[
\Fr\circ L\to\Fr\circ L\circ U\circ \Fr\simeq \Fr\circ U\circ \underline{L}\circ\Fr\to\underline{L}\circ \Fr
\]
is an equivalence, it is enough to check that this is the case after applying the forgetful functor $U\colon\Mon(\cald)\to\cald$. But this follows from the explicit description of the free functors as 
\[
\Fr(C)\simeq \bigsqcup_n C^n/ \Sigma_n
\]
(see~\cite[Example~3.1.3.12]{HA}) and by unraveling the definitions of $\underline{L}$ and the adjunction morphisms.

To prove the second statement we first remark that $\underline{R}$ has a left adjoint since it preserves all limits and filtered colimits which are formed in the underlying $\infty$-category. Moreover, any such left adjoint has to make diagram \eqref{dia:L} commute since this is the case for the corresponding diagram of right adjoints. By the above, this left adjoint has to coincide with  $\Mon(L)$. The proof for the case of groups is completely parallel.
\end{proof}

This lemma can be applied to adjunctions between cartesian closed presentable $\infty$-categories.

\begin{lem}\label{symadjoints}
Let $\calc$ and $\cald$ be closed symmetric monoidal presentable $\infty$-categories, let $L\colon \calc \to \cald$ be a symmetric monoidal left adjoint functor and let $R\colon \cald \to \calc$ be right adjoint to~$L$. 
\begin{enumerate}
\item The functors $\underline{R}\colon \Mon(\cald) \to \Mon(\calc)$ and $\underline{R}\colon \Grp(\cald) \to \Grp(\calc)$ have canonical lax symmetric monoidal structures.
\item If $\calc$ and $\cald$ are cartesian closed, then the canonical extensions $\underline{L}\colon\Mon(\calc)\to\Mon(\cald)$ and $\underline{L}\colon\Grp(\calc)\to\Grp(\cald)$ both admit structures of symmetric monoidal functors which are determined up to a contractible space of choices by the fact that the following diagrams commute:
\[
\xymatrix{
\calc\ar[r]\ar[d]_L& \Mon(\calc)\ar[d]^{\underline{L}}&&
\calc\ar[r]\ar[d]_L& \Grp(\calc)\ar[d]^{\underline{L}}\\
\cald\ar[r]&\Mon(\cald),&&
\cald\ar[r]&\Grp(\cald)\,\, .
}
\] 
\end{enumerate}
\end{lem}

\begin{proof}
\autoref{cor:monstab} tells us that $\Mon(L)$ is canonically symmetric monoidal, and the right adjoint of a symmetric monoidal functor always inherits a canonical lax symmetric monoidal structure \cite[Corollary 8.3.2.7]{HA}. Together with \autoref{leftright} this establishes the first part. The second part is an immediate consequence of \autoref{cor:monstab} and \autoref{leftright}, and again the case of groups is entirely analogous.
\end{proof}

\begin{lem}\label{factorsym}
Let $F\colon \calc \to \cald$ be an accessible functor between presentable $\infty$-categories.
\begin{enumerate}
\item We can factor $F \simeq L \circ R$ where $R$ is a right adjoint and $L$ is a left adjoint functor.
\item If $\calc$ and $\cald$ are closed symmetric monoidal, then the factorization can be chosen such that $L$ and the left adjoint to $R$ are symmetric monoidal (this means of course that the intermediate $\infty$-category is symmetric monoidal as well). In particular, $R$ itself is lax symmetric monoidal.  
\item
If $F$ preserves products and $\cald$ is cartesian closed, then
$L$ can be chosen to preserve products. 
\end{enumerate}
\end{lem}
\begin{proof}
Choose $\kappa$ sufficiently large such that both $\calc$ and $\cald$ are $\kappa$-compactly generated and $F$ preserves $\kappa$-filtered colimits.
Then the restricted Yoneda embedding $R\colon\calc\to\calp(\calc^\kappa)$ preserves limits and $\kappa$-filtered colimits, and therefore admits a left adjoint.
Similarly, the functor $L\colon\calp(\calc^\kappa)\to\cald$ induced (under colimits) by the composite $\calc^\kappa\to\calc\to\cald$ preserves all colimits, and therefore admits a right adjoint.
Since $F$ is equivalent to the composite $L\circ R$, this completes the proof of the first claim.

Now, if in addition $\calc$ and $\cald$ are closed symmetric monoidal, then it follows from the universal property of the convolution product \cite[Proposition 6.3.1.10]{HA} that~$L$ is symmetric monoidal and also that the left adjoint $\calp(\calc^\kappa)\to\calc$ of $R$ is symmetric monoidal, completing the proof of the second claim (the fact that $R$ is lax symmetric monoidal again follows from \cite[Corollary 8.3.2.7]{HA}).

Finally, if $F$ preserves products, then $L$ preserves products of representables $\calc^\kappa$, and if $\cald$ is cartesian closed then products commute with colimits in both variables.
Hence $L$ preserves products.
\end{proof}

\begin{prop}
Let $\calc$ and $\cald$ be closed symmetric monoidal presentable $\infty$-categories and let $F\colon\calc\to\cald$ be product-preserving, symmetric monoidal, and accessible. If $\cald$ is also cartesian closed then the functors $\underline{F}\colon\Mon(\calc)\to\Mon(\cald)$ and $\underline{F}\colon\Grp(\calc)\to\Grp(\cald)$ admit lax symmetric monoidal structures.
\end{prop}
\begin{proof}
Factor $F$ according to \autoref{factorsym} and apply \autoref{symadjoints}.
\end{proof}

\begin{cor}\label{cor:lax}
Let $\calc$ and $\cald$ be cartesian closed presentable $\infty$-categories and let $F\colon\calc\to\cald$ be product-preserving and accessible. Then the canonical extensions $\underline{F}\colon \Mon(\calc) \to \Mon(\cald)$ and $\underline{F}\colon \Grp(\calc) \to \Grp(\cald)$ 
are lax symmetric monoidal.
\end{cor}

\section{\texorpdfstring{$\infty$}{Infinity}-categories of semirings and rings}\label{sec:ring}

In this section we will use the results of \S\ref{sec:spec} to define and study semiring (a.k.a.\ `rig') and ring objects in suitable $\infty$-categories. We know by \autoref{thm:tensorproducts} that given a closed symmetric monoidal presentable $\infty$-category~$\calc$, there are canonical closed symmetric monoidal structures on $\Mon(\calc)$ and $\Grp(\calc)$ which will respectively be denoted by
\[
\Mono(\calc)\qquad\text{and}\qquad\Grpo(\calc).
\]

\begin{defn}
Let $\calc$ be a closed symmetric monoidal presentable $\infty$-category and let $\calo$ be an $\infty$-operad. The $\infty$-category $\Rig_\calo(\calc)$ of $\calo$-\emph{semirings} in $\calc$ and the $\infty$-category $\Ring_\calo(\calc)$ of $\calo$-\emph{rings} in $\calc$ are respectively defined as the $\infty$-categories of $\calo$-algebras
\begin{equation*}
\Rig_\calo(\calc) := \Alg_\calo(\Mono(\calc)) \qquad \text{and}\qquad \Ring_\calo(\calc) := \Alg_\calo(\Grpo(\calc)).
\end{equation*}
\end{defn}

In the case of ordinary categories and the associative or commutative operad, the alternative terminology \emph{rig objects} is also used for what we call semiring objects, hence the notation. We will be mainly interested in the case of $\calo = \mathbb{E}_n$ for $n=1,2,\ldots,\infty$. In the case $n=1$, $\Ring_{\bbE_1}(\calc)$ is the $\infty$-category of \emph{associative rings} in $\calc$ and, in the case $n=\infty$, $\Ring_{\bbE_\infty}(\calc)$ is the $\infty$-category of \emph{commutative rings} in $\calc$. Similarly, there are $\infty$-categories of associative or commutative semirings in~$\calc$.

Let us take up again the examples of \S\ref{sec:spec}.

\begin{eg}
\begin{enumerate}
\item In the special case of the cartesian closed presentable $\infty$-category $\Set$ of sets, our notion of associative or commutative (semi)ring object coincides with the corresponding classical notion. 
\item Since the $\infty$-category~$\cals$ of spaces is cartesian closed and presentable, we obtain, for each $\infty$-operad~$\calo$, the $\infty$-category $\Rig_\calo(\cals)$ of $\calo$\emph{-rig spaces} and the $\infty$-category $\Ring_\calo(\cals)$ of $\calo$\emph{-ring spaces}. For the special case of the operads $\calo=\mathbb{E}_n$ for $n=1,\ldots,\infty$, the point-set analogue of these spaces were intensively studied by May and others using carefully chosen pairs of operads
(see the recent articles \cite{May_WhatI, May_WhatII, May_WhatIII} and the many references therein).
\item In the case of the cartesian closed presentable $\infty$-category $\Cat$ of ordinary small categories, we obtain the $\infty$-category $\RigoCat$ of $\calo$\emph{-rig categories} and the $\infty$-category $\RingoCat$ of $\calo$\emph{-ring categories}.
Coherences for `lax' semiring categories have been studied by Laplaza \cite{Laplaza2}, \cite{Laplaza1}; note that, in our case, all coherence morphisms must be invertible.
It should be possible to obtain a precise comparison of our notion with these more classical ones, 
but we bypass this via a recognition principle (\autoref{cor:awesome}) for semiring $\infty$-categories which allows us to work directly with the examples of interest to us, without having to check coherences for distributors.
\item An $\infty$-categorical version of the previous example is obtained by considering the cartesian closed presentable $\infty$-category $\Catoo$. Associated to it there is the $\infty$-category $\RigoCat_\infty$ of $\calo$\emph{-semiring $\infty$-categories} and the $\infty$-category $\RingoCat_\infty$ of $\calo$\emph{-ring $\infty$-categories}.
\end{enumerate}
\end{eg}

\begin{rmk}
For a general closed symmetric monoidal presentable $\infty$-category $\calc$ there are two potentially different symmetric monoidal structures playing a role in the notion of an $\calo$-(semi)ring object. Thus it may be useful to provide an informal description of the structure given by an $\Eoo$-semiring object in~$\calc$. It consists of an object~$R\in\calc$ together with an addition map $+\colon R\times R\to R$ and a multiplication map $\times\colon R\otimes R\to R$ such that both maps are coherently associative and commutative. Moreover, the multiplication has to distribute over the addition in a homotopy coherent fashion. In the case of an ordinary category with the Cartesian monoidal structure, our notion reduces to the usual one.
\end{rmk}

Similarly to the case of commutative monoids and commutative groups, \autoref{thm:tensorproducts} also guarantees that the $\infty$-category $\Sp(\calc)$ of spectrum objects associated to a closed symmetric monoidal presentable $\infty$-category~$\calc$ has a canonical closed symmetric monoidal structure $\Sp^\otimes(\calc)$. This allows us to make the following definition.

\begin{defn}
Let $\calc$ be a closed symmetric monoidal presentable $\infty$-category and let $\calo$ be an $\infty$-operad. The $\infty$-category $\RingSp_\calo(\calc)$ of $\calo$-\emph{ring spectrum objects in}~$\calc$ is defined as
\begin{equation*}
\RingSp_\calo(\calc) := \Alg_\calo(\Sp^\otimes(\calc)).
\end{equation*}
\end{defn}

\begin{thm}\label{thm:rigmon}
Let $\calc$ be a closed symmetric monoidal presentable $\infty$-category and let~$\calo$ be an $\infty$-operad. Then the group completion functor $\Mon(\calc) \to \Grp(\calc)$ and the 
associated spectrum functor $\Grp(\calc) \to \Sp(\calc)$ refine to functors
\begin{equation*}
\Rig_\calo(\calc) \to \Ring_{\calo}(\calc)  \qquad\text{and}\qquad \Ring_\calo(\calc) \to \RingSp_\calo(\calc), 
\end{equation*}
called the \emph{ring completion} and the \emph{associated ring spectrum functor}, respectively.
\end{thm}
\begin{proof}
This is clear since the group completion $\Mon(\calc) \to \Grp(\calc)$ and also the associated spectrum functor $\Grp(\calc) \to \Sp(\calc)$ are symmetric monoidal as shown in \autoref{thm:tensorproducts}.
\end{proof}

\begin{eg}
\begin{enumerate}
\item In the special case of the $\infty$-category $\Set$ of sets this reduces to the usual ring completion of associative or commutative semirings. 
\item Given an $\infty$-operad~$\calo$, we obtain an associated ring completion functor $\Rig_\calo(\cals)\to\Ring_\calo(\cals)$ from $\calo$-rig spaces to $\calo$-ring spaces and an associated ring spectrum functor $\Ring_\calo(\cals)\to\RingSp_\calo(\cals)$ from $\calo$-ring spaces to $\calo$-ring spectra. The latter $\infty$-category will also be written $\RingSp_\calo.$
\item Let us again consider the cartesian closed presentable $\infty$-category $\Cat$ of ordinary small categories. Then for each $\infty$-operad~$\calo$, we obtain a ring completion functor $\RigoCat\to\RingoCat$ from $\calo$-rig categories to $\calo$-ring categories. 
\item Again, we immediately obtain an $\infty$-categorical refinement of the previous example. For each $\infty$-operad~$\calo$, we obtain a ring completion functor $\RigoCat_\infty\to\RingoCat_\infty$ from $\calo$-rig $\infty$-categories to $\calo$-ring $\infty$-categories. Using explicit models, a similar construction was obtained by Baas-Dundas-Richter-Rognes in~\cite{BDRR}.
\end{enumerate}
\end{eg}

\autoref{thm:rigmon} shows that semirings can be used to produce highly structured ring spectra. Unfortunately, the definition of a semiring object is a bit indirect, so in practice it is often difficult to write down explicit examples of such objects. \autoref{seminringex} provides a natural class of semirings in the case of the cartesian closed $\infty$-category $\calc = \Cat_\infty$. Moreover, this is the class that is of most interest in applications to algebraic K-theory, as we discuss in \S\ref{sec:infinite}.

We conclude this section with a base-change result (similar to \autoref{cormon}) which sheds some light on the definition of semiring and ring object. This result will also be needed in \autoref{sec:app} where we show $\bbE_n$-(semi)rings to be \emph{algebraic}. 

\begin{prop}\label{prop_rig_ten}
Let $\calc$ be a cartesian closed presentable $\infty$-category and $\calo$ an $\infty$-operad. Then we have equivalences
\begin{equation*}
\Rig_\calo(\calc) \simeq \calc \otimes \Rig_\calo(\cals) \qquad \text{and} \qquad  \Ring_\calo(\calc) \simeq \calc \otimes \Ring_\calo(\cals).
\end{equation*}
\end{prop}

\begin{proof}
We show more generally that, for $\cald$ any closed symmetric monoidal presentable $\infty$-category, there exists a canonical equivalence
\begin{equation}\label{eqn:algbasechange}
\Alg_\calo(\cald)\otimes\calc\to\Alg_\calo(\cald\otimes\calc).
\end{equation}
Then, taking $\cald$ to be $\Mon(\cals)$, using Theorem \ref{thm:idem}, we obtain the desired chain of equivalences
\[
\Rig_\calo(\calc)\simeq\Alg_\calo(\Mon(\calc))\simeq\Alg_\calo(\Mon(\cals)\otimes\calc)
\simeq\Alg_\calo(\Mon(\cals))\otimes\calc\simeq\Rig_\calo(\cals)\otimes\calc\, .
\]
In the case of rings we get an analogous chain of equivalences.

To show \eqref{eqn:algbasechange}, first consider the case in which $\calc=\calp(\calc_0)$ is the $\infty$-category of presheaves of spaces on a (small) $\infty$-category $\calc_0$.
In this case, we have that $\cald\otimes\calc\simeq\Fun(\calc^{op}_0,\cald)$, so that
\[
\Alg_\calo(\cald)\otimes\calc\simeq\Fun(\calc^{op}_0,\Alg_\calo(\cald))\simeq\Alg_\calo(\Fun(\calc^{op}_0,\cald))\simeq\Alg_\calo(\cald\otimes\calc).
\]
A general cartesian closed presentable $\infty$-category $\calc$ is a full symmetric monoidal subcategory of some $\calp(\calc_0)$, say for $\calc_0$ the full subcategory of $\kappa$-compact objects in $\calc$ for a sufficiently large regular cardinal $\kappa$.
Since $\cald\otimes\calc\simeq\FunR(\calc^{op},\cald)$, we see that $\cald\otimes\calc$ is a full symmetric monoidal subcategory of $\cald\otimes\calp(\calc_0)$, and similarly with $\cald$ replaced by $\Alg_\calo(\cald)$.
Thus it suffices to show that $\Alg_\calo(\cald)\otimes\calc$ and $\Alg_\calo(\cald\otimes\calc)$ define equivalent full subcategories of $\Alg_\calo(\cald)\otimes\calp(\calc_0)\simeq\Alg_\calo(\cald\otimes\calp(\calc_0))$.

If $\calo$ is monochromatic (i.e. if there exists an essentially surjective functor $\Delta^0\to\calo^\otimes_{\langle 1\rangle}$), then an object of $\Alg_\calo(\cald\otimes\calp(\calc_0))$ lies in the full subcategory $\Alg_\calo(\cald\otimes\calc)$ if and only if the projection to $\cald\otimes\calp(\calc_0)$ factors through $\cald\otimes\calc$.
For arbitrary $\calo$, an object of $\Alg_\calo(\cald\otimes\calp(\calc_0))$ lies in the full subcategory $\Alg_\calo(\cald\otimes\calc)$ precisely when the restriction along any full monochromatic suboperad $\calo'\to\calo$ satisfies this same condition.
As the analogous results for $\Alg_\calo(\cald)\otimes\calc$ hold by the same argument, we see that $\Alg_\calo(\cald)\otimes\calc$ and $\Alg_\calo(\cald\otimes\calc)$ define equivalent full subcategories of $\Alg_\calo(\cald)\otimes\calp(\calc_0)\simeq\Alg_\calo(\cald\otimes\calp(\calc_0))$.
\end{proof}

\section{Multiplicative infinite loop space theory}\label{sec:infinite}

In this section we apply the results of the previous section to some specific $\infty$-categories; namely, we consider the $\infty$-categories $\cals$ of spaces, the $\infty$-category $\Cat$ of  ordinary categories (really a 2-category, but we regard it as an $\infty$-category), and the $\infty$-category $\Catoo$  of $\infty$-categories. Let us emphasize that, as a special case of \autoref{thm:rigmon}, the group completion and the associated spectrum functor
\[
\Mon(\cals) \to \Grp(\cals) \to \Sp
\]
refine to functors
\[
\Rig_\calo(\cals) \to \Ring_\calo(\cals) \to \RingSp_\calo.
\] 
This gives us not only a way of obtaining (highly structured) ring spectra, but it also allows us to identify certain spectra as ring spectra.

Recall that the group completion functor $\Mon(\cals)\to\Grp(\cals)\to\Sp$ plays an important role in algebraic K-theory. The input data for algebraic K-theory is often a symmetric monoidal category $\calm$; as a primary example, we have the category $\calm=\Proj_R$ of finitely generated projective modules over a ring $R$, which is symmetric monoidal under the direct sum $\oplus$, the coproduct. In any case, given such a category~$\calm$, we form the subcategory of isomorphisms $\calm^\sim$ and pass to the geometric realization~$|\calm^\sim|$.
That way we obtain an $\bbE_\infty$-space~$|\calm^\sim|$, i.e., an object of $\Mon(\cals)$. The \emph{algebraic K-theory spectrum} $\K (\calm)$ is then defined to be the spectrum associated to the group completion of $|\calm^\sim|$, see e.g.\ \cite{segal_categories}. In other words, (direct sum) algebraic K-theory is defined as the composition
\begin{equation}\label{ktheory}
\K\colon  \SymMonCat \xrightarrow{(-)^\sim} \SymMonCat \xrightarrow{|-|} \Mon(\cals) \to \Grp(\cals) \to \Sp .
\end{equation}
\noindent
It is a result of May \cite{May82}, with refinements by Elmendorf-Mandell \cite{EM06} and Bass-Dundas-Richter-Rognes \cite{BDRR}, that this functor respects multiplicative structures, in the appropriate sense. Our methods give an even more refined result.

\begin{prop}
The algebraic $\K$-theory functor $\K\colon  \SymMonCat\to\Sp$ is lax symmetric monoidal.
In particular, it induces a functor $\Rig_\calo\Cat \to \RingSp_\calo$ for any $\infty$-operad $\calo$. 
\end{prop}
\begin{proof}
The last two functors in the composition \eqref{ktheory} are symmetric monoidal by \autoref{thm:tensorproducts}.
The remaining two functors $(-)^\sim\colon\SymMonCat\to\SymMonCat$ and $|-|\colon \SymMonCat \to \Mon(\cals)$ are the canonical extensions of the product preserving functors $(-)^\sim\colon\Cat\to\Cat$ and $|-|\colon \Cat \to \cals$ respectively. Since these latter functors are accessible, \autoref{cor:lax} implies that their canonical extensions are lax symmetric monoidal, concluding the proof.
\end{proof}

We now have the tools necessary to establish corresponding results in the $\infty$-categorical case.
Note that the composition of the first two functors in \eqref{ktheory} is the same as the composition of the nerve $\SymMonCat \to \SymMonCat_\infty$ followed by the functor $(-)^\sim\colon\SymMonCat_\infty \to \Mon(\cals)$, which sends a symmetric monoidal $\infty$-category to its maximal subgroupoid, and of course is again symmetric monoidal. This allows us to recover the algebraic K-theory of a symmetric monoidal category $\calm$ by an application of the following $\infty$-categorical version of algebraic K-theory to the nerve of $\calm$.

\begin{defn}
Let $\calm$ be a symmetric monoidal $\infty$-category. The \emph{algebraic $\mathrm{K}$-theory spectrum} $\K(\calm)$ is the spectrum associated to the group completion of $\calm^\sim$. Thus, the algebraic K-theory functor is defined as the composition
\begin{equation}\label{ktheory2}
\K\colon  \SymMonCat_\infty \xrightarrow{(-)^\sim} \Mon(\cals) \longrightarrow \Grp(\cals) \longrightarrow \Sp.
\end{equation}
\end{defn}

\begin{rmk}
Strictly speaking, this is the {\em direct sum} K-\emph{theory}, since it does not take into account a potential exactness (or Waldhausen) structure on the symmetric monoidal $\infty$-categories in question.
Nevertheless, in many cases of interest, e.g.\ that of a connective ring spectrum $R$, the algebraic K-theory of $R$, defined in terms of Waldhausen's $S_\bullet$ construction applied to the stable $\infty$-category of $R$-modules (which agrees with the K-theory of any suitable model category of $R$-modules, see~\cite{BGT1} for details), is computed as the direct sum K-theory of the symmetric monoidal $\infty$-category $\mathrm{Proj}_R$ of finitely-generated projective $R$-modules (\cite[Definition 8.2.2.4]{HA}).

For more sophisticated versions of K-theory, the situation is slightly more complicated but entirely analogous. In \cite{BGT2} it is shown that the algebraic K-theory $\mathrm{K}\colon\Catoo^\mathrm{perf}\to\Sp$ of small idempotent-complete stable $\infty$-categories is a lax symmetric monoidal functor, as is the nonconnective version; the methods employed to do so are similar to the ones used in the present paper, in that $\mathrm{K}$ is shown to be the tensor unit in a symmetric monoidal $\infty$-category of all additive (respectively, localizing) functors $\Catoo^\mathrm{perf}\to\Sp$, so that the commutative algebra structure ultimately relies on the existence of an idempotent object in an appropriate symmetric monoidal $\infty$-category.
The case of general Waldhausen $\infty$-categories is treated in \cite{Bar}, where it is shown that the algebraic K-theory $\mathrm{K}\colon\mathrm{Wald}_\infty\to\Sp$ of Waldhausen $\infty$-categories is again a lax symmetric monoidal functor.
\end{rmk}

As already mentioned, the $\infty$-categorical algebraic K-theory 
\[
\K\colon\SymMonCatoo\to\Sp
\]
applied to nerves of ordinary symmetric monoidal categories recovers the 1-categorical algebraic K-theory $\K\colon\SymMonCat\to\Sp.$ Note, however, that the inclusion of symmetric monoidal 1-categories into symmetric monoidal $\infty$-categories given by the nerve functor does not commute with the tensor products. In fact, the tensor product $\mathrm{N}(\calc)\otimes\mathrm{N}(\cald)$ of the nerves of two symmetric monoidal 1-categories $\calc$, $\cald$ need not again be (the nerve of) a symmetric monoidal 1-category; rather, one can show that $\mathrm{N}(\calc\otimes\cald)$ is the 1-categorical truncation of $\mathrm{N}(\calc)\otimes\mathrm{N}(\cald)$.

\begin{thm}
The algebraic $\K$-theory functor $\K\colon \SymMonCat_\infty \to \Sp$ is lax symmetric monoidal. In particular, it refines to a functor $\Rig_\calo(\Cat_\infty) \to \RingSp_\calo$ for any $\infty$-operad $\calo$. 
\end{thm}
\begin{proof}
The proof is almost the same as in the 1-categorical case. The last two functors in the defining composition \eqref{ktheory2} are symmetric monoidal by \autoref{thm:tensorproducts}.
The remaining functor $(-)^\sim\colon\SymMonCatoo\to\Mon(\cals)$ is the canonical extension of the accessible, product preserving functors $(-)^\sim\colon\Catoo\to\cals$. Thus, \autoref{cor:lax} implies that this canonical extension is lax symmetric monoidal as intended.
\end{proof}

\begin{rmk}
The $\K$-theory functor is defined as the composition \eqref{ktheory2} of lax symmetric monoidal functors. We know that the last two of these (namely, the group completion and the associated spectrum functor) are actually symmetric monoidal. Thus, one might wonder whether also the first functor (and hence the $\K$-theory functor) is symmetric monoidal as well. This is not the case, as the following counterexample shows.

Let us begin by recalling from \cite[Remark~2.1.3.10]{HA} that the $\infty$-category $\mathrm{Mon}_{\bbE_0}(\Cat_\infty)$ is equivalent to $(\Cat_\infty)_{\Delta^0/}.$ Thus, an object in $\mathrm{Mon}_{\bbE_0}(\Cat_\infty)$ is just an $\infty$-category~$\calc$ together with a chosen object $x\in\calc.$ The fact that an ordinary monoid gives rise to a category with one object (which is hence distinguished) admits the following $\infty$-categorical variant. There is a functor
\[
\mathrm{B}\colon\mathrm{Mon}_{\bbE_1}(\cals)\to\mathrm{Mon}_{\bbE_0}(\Cat_\infty)
\]
which is left adjoint to the functor which sends $x\colon\Delta^0\to\calc$ to the endomorphism monoid~$\mathrm{End}_{\calc}(x)$ of the distinguished object. Similarly, there is a functor
\[
\mathrm{B}\colon\mathrm{Mon}_{\bbE_\infty}(\cals)\to\mathrm{Mon}_{\bbE_\infty}(\Cat_\infty)
\]
which is left adjoint to the functor which sends a symmetric monoidal $\infty$-category to the $\Eoo$-monoid of endomorphisms of the monoidal unit (we are also using the fact that $\bbE_n\otimes\bbE_\infty\simeq\bbE_\infty$ for $n=0,1$). 
 
Now, let $\calf=\Fr(\Delta^0)$ denote the free symmetric monoidal $\infty$-category on the point, which is to say the nerve of the groupoid of finite sets and isomorphisms. We claim that, for any symmetric monoidal $\infty$-groupoid $\calc$,
\[
(\mathrm{B}\calf)\otimes\calc\simeq\mathrm{B}\calc.
\]
This is clearly true if $\calc=\calf$, and the general formula follows by the observation that both sides commute with colimits in the $\calc$ variable and the fact that every symmetric monoidal $\infty$-groupoid is an iterated colimit of $\calf$.
But the groupoid core $(\mathrm{B}\calf)^\sim$ is trivial.
Thus, $\K(\mathrm{B}\calf)\otimes\K(\calc)=0$ for every $\calc$.
On the other hand, taking $\calc=\mathbb{Z}$, we have that $(\mathrm{B}\calc)^\sim\simeq\mathrm{B}\calc$, so $\K(\mathrm{B}\calc)\simeq\Sigma\mathrm{H}\mathbb{Z}$, the suspension of the Eilenberg-MacLane spectrum.
\end{rmk}

We have the following `recognition principle' for semiring $\infty$-categories.

\begin{thm}\label{seminringex}
 Let~$\calc$ be an $\bbE_n$-monoidal $\infty$-category with coproducts such that the monoidal product
 \[
 \otimes\colon\calc\times\calc\to\calc
 \]
 preserves coproducts separately in each variable. 
Then~$(\calc,\sqcup,\otimes)$ is canonically an object 
of $\Rig_{\bbE_n}(\Catoo)$. 
\end{thm}

\begin{proof}
Let $\Cat_\infty^\Sigma$ be the $\infty$-category of $\infty$-categories which admit finite coproducts and coproduct preserving functor. There is a fully-faithful functor 
\begin{equation*}
  \Cat_\infty^{\Sigma} \to \SymMonCat_\infty 
\end{equation*}
given by considering an $\infty$-category with coproducts as a cocartesian symmetric monoidal $\infty$-category (see \cite[Variant~2.4.3.12]{HA}). We want to show that this functor naturally extends to a lax symmetric monoidal functor, essentially by the construction of the tensor product on $\Cat_\infty^\Sigma$ of \cite[Proposition 6.3.1.10]{HA} . From this the claim follows, since an $\bbE_n$-algebra in $\Cat_\infty^\Sigma$ is the same as an 
$\bbE_n$-monoidal $\infty$-category such that the tensor product preserves finite coproducts in each variable separately.

The first thing we want to observe is that the $\infty$-category $\Cat_\infty^{\Sigma}$ is preadditive.
To see this, note that $\Cat_\infty^{\Sigma}$ has finite coproducts and products, because $\Cat_\infty^{\Sigma}$ is presentable (this follows from \cite[Lemma 6.3.4.2]{HA} by taking $\mathcal{K}$ to be the collection of finite sets).
It remains to check that the product $\calc\times\cald$, which is calculated as the product in $\Ho(\Cat_\infty)$, satisfies the universal property of the coproduct in $\Ho(\Cat_\infty^\Sigma)$.
Given a third $\infty$-category with finite coproducts $\cale$, we note that any pair of coproduct preserving functors $f\colon\calc\to\cale$ and $g\colon\cald\to\cale$ extends to the coproduct preserving functor
\[
\calc\times\cald\overset{f\times g}{\longrightarrow}\cale\times\cale\overset{\sqcup}{\longrightarrow}\cale.
\] 
Moreover, this extension is unique up to homotopy, because $(c,d)\cong(c,\emptyset)\sqcup(\emptyset,d)$ for any $(c,d)\in\calc\times\cald$.

Using \cite[Proposition 6.3.1.10]{HA} again, the inclusion functor $i\colon \Cat_\infty^\Sigma \to \Cat_\infty$ admits a left adjoint $L$ which is symmetric monoidal. By \autoref{universalmonoidal} the functor $L$ extends to a left adjoint functor
\begin{equation*}
L'\colon \SymMonCat_\infty \simeq \Mon(\Cat_\infty) \to \Mon(\Cat_\infty^\Sigma)\simeq\Cat_\infty^\Sigma .
\end{equation*}
The right adjoint of this functor can be described as the functor 
\begin{equation*}
\Mon(i)\colon \Cat_\infty^\Sigma \simeq \Mon(\Cat_\infty^\Sigma) \to \Mon(\Cat_\infty).
\end{equation*}

We can now conclude that $\Mon(i)$ is lax symmetric monoidal since it is right adjoint to a symmetric monoidal functor. It remains to show that $\Mon(i)$ is the desired functor. 
This is obvious. 
\end{proof}

\begin{rmk}
Preserving coproducts is a condition!
\end{rmk}

\begin{cor}\label{cor:awesome}
If $\calc$ is an ordinary monoidal category with coproducts such that $\otimes\colon\calc\times\calc\to\calc$ preserves coproducts in each variable separately. 
Then~$(\calc,\sqcup,\otimes)$ is canonically an object of $\Rig_{\bbE_1}(\Cat) \subset\Rig_{\bbE_1}(\Catoo)$. If $\calc$  is moreover braided or symmetric monoidal then $(\calc,\sqcup,\otimes)$  is an object of $\Rig_{\bbE_2}(\Cat)$ or $\Rig_{\bbE_\infty}(\Cat)$ respectively.
\end{cor}
\begin{proof}
We only need the identification of the $\bbE_n$-monoids in $\Cat$ with the respective monoidal categories. This has been given in \cite[Example 5.1.2.4]{HA}.
\end{proof}

\begin{cor}
Let~$\calc$ be an $\bbE_n$-monoidal $\infty$-category with coproducts such that~$\otimes\colon\calc\times\calc\to\calc$ preserves coproducts in each variable separately. 
Then the largest Kan complex $\calc^\sim$ inside of $\calc$ together with $\sqcup$ and $\otimes$ is an object of $\Rig_{\bbE_n}(\cals)\subseteq\Rig_{\bbE_n}(\Catoo).$
\end{cor}
\begin{proof}
The functor $(-)^\sim\colon\Catoo \to \cals\subset\Catoo$ preserves products and is accessible. Thus we can apply \autoref{cor:lax} to deduce that the induced functor 
$\Mon(\Catoo) \to \Mon(\Catoo)$ is lax symmetric monoidal. But this implies that we obtain a further functor $\Rig_{\bbE_n}(\Catoo) \to \Rig_{\bbE_n}(\Catoo)$ which preserves the underlying object of $\Catoo$. 
Now apply this functor to the semiring $\infty$-category of \autoref{seminringex}.
\end{proof}

\begin{eg}
\begin{enumerate}
\item For an ordinary commutative ring $R$, let $\Mod_R$ denote the (ordinary) category of $R$-modules. Then $\Mod_R$ and the $\infty$-groupoid $\Mod_R^\sim$, equipped with the operations $\oplus$ and $\otimes_R$, form $\Eoo$-semiring categories. The same applies to the category of sheaves on schemes and other similar variants.

\item For an $\bbE_n$-ring spectrum $R$, the $\infty$-category $\Mod_R$ of (left) $R$-modules is a $\bbE_{n-1}$-monoidal $\infty$-category by \cite[Section 6.3 or Proposition 8.1.2.6]{HA}. Since the tensor product preserves coproducts in each variable we conclude that
 $\Mod_R$, together with the coproduct $\oplus$ and tensor product $\otimes_R$, is an $\bbE_{n-1}$-semiring $\infty$-category.
\end{enumerate}
\end{eg}

Now we want to apply this to identify certain spectra as $\Eoo$-ring spectra. For a connective $\mathbb{E}_{n+1}$-ring spectrum $R$ the $\infty$-category $\Proj_R$ of finitely generated projective $R$-modules is an $\bbE_n$-semiring. The K-theory spectrum $\K(R)$ can then be defined as $\K(\Proj_R)$.
This definition is actually equivalent to the definition using Waldhausen categories: for the variant which uses finitely generated free $R$-modules in place of projective, this is shown in \cite[Chapter VI.7]{EKMM}, and for the general case this follows from \cite[Section 4]{BGT}.

\begin{cor}
For a connective $\bbE_{n+1}$-ring spectrum $R$ the algebraic $\K$-theory spectrum $\K(R)$ is an $\bbE_{n}$-ring spectrum.
\end{cor}

We also have the following proposition, which states roughly that group completion of monoidal $\infty$-categories not only inverts objects, but arrows as well. It also shows why it is necessary to discard all non-invertible morphisms \emph{before} group completion.

\begin{prop}
The underlying $\infty$-category of an $\bbE_\infty$-group object of $\Cat_\infty$ is an $\infty$-groupoid. More precisely, the group completion functor $\Mon(\Cat_\infty)\to\Grp(\Cat_\infty)$
factors through the groupoid completion
\[
\Mon(\Cat_\infty)\to\Mon(\cals)\to\Grp(\cals)\to\Grp(\Cat_\infty)
\]
and induces an equivalence $\Grp(\cals)\simeq\Grp(\Cat_\infty)$.
\end{prop}

\begin{proof}
Let $\calc$ be an  $\bbE_\infty$-group object of $\Cat_\infty$. Then the underlying $\infty$-category of $\calc$ is an $\infty$-groupoid precisely if its homotopy category $\Ho(\calc)$ is an ordinary groupoid. Thus 
it suffices to show that $\Ho(\calc)$ is a groupoid.
But since $\Ho(\calc)$ is a group object in $\Cat$, this reduces the proof of the 
proposition to ordinary categories $\calc$.

A group object $\calc$ in categories is a symmetric monoidal category $(\calc,\otimes)$ together with an `inversion' functor $I\colon \calc \to \calc$
as in to \autoref{prop:gpl}. We clearly have $I^2 \simeq \id$. 
As a first step we show that all endomorphisms of the tensor unit $\bbone$ in $\calc$ are automorphisms. This follows from the Eckman-Hilton argument since $\hom_\calc(\bbone,\bbone)$ carries two commuting monoid structures (composition and tensoring), and as one of these is a group structure the other must also be as well. It follows that all endomorphisms in $\calc$ are automorphisms by the identification
$$I(x) \otimes -\colon \hom_\calc(x,x) \cong \hom_\calc(\bbone,\bbone).$$
Finally, to show that $\calc$ is a groupoid, it now suffices to show that for every morphisms $f\colon x \to y$ in $\calc$ there is a morphism $g: y \to x$ in $\calc$. By tensoring with $I(y)$ we see that we may assume that $y = \bbone$. Then we have $I(f)\colon I(x) \to \bbone$, and therefore, using the usual identifications, $g:=I(f) \otimes x\colon \bbone \to x$.
\end{proof}

\appendix

\section{Comonoids}\label{comonoids}
 
In this short section we establish additional universal mapping properties for $\Mon(\cals)$ and $\Grp(\cals)$ respectively. This gives a characterization of these $\infty$-categories among all presentable $\infty$-categories and not only among the (pre)additive ones.

 Let us denote by $\FunRAd(\calc,\cald)$ the $\infty$-category of right adjoint functors from $\calc$ to $\cald$, which is a full subcategory of $\Fun(\calc,\cald)$.

\begin{lem}
If $\calc$ and $\cald$ are presentable, then we have canonical equivalences
\[
\Mon\big(\FunRAd(\calc,\cald)\big) \simeq \FunRAd\big(\calc, \Mon(\cald) \big) \quad\text{and}\quad \Grp\big(\FunRAd(\calc,\cald)\big) \simeq \FunRAd\big(\calc, \Grp(\cald) \big).
\]
\end{lem}
\begin{proof}
We note that right adjoint functors between presentable $\infty$-categories can be described as accessible functors that preserve limits. Then then the proof works exactly the same as the proof of Lemma \ref{lem:algR}.
\end{proof}

\begin{defn}
Let $\calc$ be an $\infty$-category with finite coproducts. We define the $\infty$-categories of \emph{comonoids} and \emph{cogroups} in~$\calc$ to be the respective $\infty$-categories
\[
\coAlg(\calc) = \Mon\big(\calc\op\big)\op \quad\text{and}\quad 
\coGrp(\calc) = \Grp\big(\calc\op\big)\op.
\]
\end{defn}

\begin{rmk}
The comonoids as defined above are comonoids for the coproduct as tensor product. This is a structure which is often rather trivial. 
For example in the $\infty$-category~$\cals$ of spaces (or in the ordinary category of sets) there is exactly one comonoid in the sense above, namely the empty set~$\emptyset$. 
\end{rmk}

\begin{prop}\label{mappingproperty}
Let~$\calc$ and~$\cald$ be presentable $\infty$-categories. Then there are natural equivalences
\[
\FunL(\Mon(\calc) , \cald ) \simeq \coAlg\big(\FunL(\calc,\cald)\big) \quad\text{and}\quad
\FunL(\Grp(\calc) , \cald ) \simeq \coGrp\big(\FunL(\calc,\cald)\big).
\]
In particular, for a presentable $\infty$-category~$\cald$ we have natural equivalences
\[\FunL(\Mon(\cals), \cald) \simeq \coAlg(\cald)\quad\text{and}\quad\FunL(\Grp(\cals), \cald) \simeq \coGrp(\cald).\] 
\end{prop}

\begin{proof}
Let us recall that given two $\infty$-categories $\cale$ and $\calf$ then there is an equivalence of categories $\FunL(\cale,\calf)$ and $\FunRAd(\calf,\cale)\op$ (\cite[Proposition~5.2.6.2]{HTT}).
The adjoint functor theorem (\cite[Corollary~5.5.2.9]{HA}) together with \autoref{lem:algR} then yield the following chain of equivalences:
\begin{eqnarray*}
\FunL(\Mon(\calc) ,\cald ) & \simeq &\FunRAd\big(\cald,\Mon(\calc)\big)\op\\
& \simeq& \Mon\big(\FunRAd(\cald,\calc)\big)\op \\
& \simeq& \Mon\big(\FunL(\calc,\cald)\op\big)\op \\
& =&  \coAlg\big(\FunL(\calc,\cald)\big).
\end{eqnarray*}
In the special case of $\calc=\cals$ we can use the universal property of $\infty$-categories of presheaves (\cite[Theorem~5.1.5.6]{HTT}) to extend the above chain of equivalences by
\[\coAlg\big(\FunL(\cals,\cald)\big)\simeq\coAlg(\cald).\]
This settles the case of monoids and the case of groups works the same. 
\end{proof}

\section{Algebraic theories and monadic functors}\label{sec:app}

\newcommand{\Algff}{\Mod^\mathrm{ff}}
\newcommand{\Modff}{\Algff}
\newcommand{\map}{\mathrm{map}}
\newcommand{\Monad}{\mathrm{Monad}}

In this section we give a short discussion of Lawvere algebraic theories in $\infty$-categories and show that our examples are algebraic. For other treatments of $\infty$-categorical algebraic theories, see \cite{cranch}, \cite{cranch2}, \cite[Section 32]{Joyal} and \cite[Section 5.5.8]{HTT}. We write $\mathcal{F}\mathrm{in}$ for the category of finite sets. 
\begin{defn}
An \emph{algebraic theory} is an $\infty$-category~$\bbT$ with finite products and a distinguished object $1_{\mathbb{T}}$, such that the unique product-preserving functor $\mathrm{N}(\mathcal{F}\mathrm{in})\op\to\mathbb{T}$
which sends the singleton to $1_\mathbb{T}$ is essentially surjective. A \emph{morphism of algebraic theories} is a functor which preserves products and the distinguished object. We write $\Th\subseteq(\Catoo^\Pi)_\ast$ for 
the $\infty$-category of theories and morphisms thereof.
\end{defn}

This is the obvious generalization of ordinary algebraic theories, as defined by Lawvere~\cite{Lawvere}, to $\infty$-categories.

\begin{defn}
 Let $\calc$ be an $\infty$-category with finite products. A \emph{model} (or an \emph{algebra}) in $\calc$ for an algebraic theory $\bbT$ is a finite product preserving functor $\bbT\to\calc$. We write $\Mod_{\bbT}(\calc)$ for the $\infty$-category of models of~$\bbT$ in~$\calc$, i.e., for the full subcategory of $\Fun(\bbT,\calc)$ spanned by the models. 
\end{defn}

If $\calc$ is a presentable $\infty$-category and $\bbT$ a theory, then $\Mod_{\bbT}(\calc)$ is again presentable. This follows since $\Mod_{\bbT}(\calc)$ is an accessible localization of the presentable $\infty$-category $\Fun(\bbT,\calc)$ (the proof is similar to the one of \autoref{prop:monpres} which takes care of the case of commutative monoids). Applying the adjoint functor theorem we also get that the forgetful functor $\Mod_{\bbT}(\calc)\to\calc$, i.e. the evaluation at the distinguished object $1_\bbT,$ has a left adjoint.

\begin{prop}\label{prop_ten_al}
Let $\calc$ be a presentable $\infty$-category and $\mathbb{T}$ a theory. Then we have an equivalence
\begin{equation*}
 \Mod_\bbT(\calc) \simeq \calc \otimes \Mod_{\bbT}(\cals).
\end{equation*}  
\end{prop}
\begin{proof}
The same proof as for \autoref{lem:algR} shows that  we have an equivalence
$$  \Mod_\bbT(\FunR(\calc^{op},\cals)) \simeq \FunR(\calc^{op}, \Mod_\bbT(\cals)).$$
This then implies the claim since we have $\calc \otimes \cald \simeq \FunR(\calc^{op},\cald)$ for any presentable $\infty$-category  $\cald$.
\end{proof}

A \emph{monad} on an $\infty$-category $\calc$ is an algebra $M$ in the monoidal $\infty$-category $\Fun(\calc,\calc)$ of endofunctors; see \cite[Chapter 6.2]{HA} for details.
Any such monad $M\in\Alg(\Fun(\calc,\calc))$ admits an $\infty$-category of \emph{modules} which we denote $\Mod_M(\calc)$. This $\infty$-category comes equipped with a forgetful functor $\Mod_M(\calc) \to \calc$ which is  a right adjoint. 
Thus, given an arbitrary right adjoint functor $U\colon \cald \to \calc$, it is natural to ask whether this functor is equivalent to the forgetful functor from modules over a monad on $\calc$. In this case the corresponding monad is uniquely determined as the composition $M = U \circ F$, where $F$ is a left adjoint of~$U$. The functors $U$ for which this is the case are called 
\emph{monadic}. 

The Barr-Beck theorem (also called Beck's monadicity theorem) gives necessary and sufficient conditions for a functor $U$ to be monadic. The conditions are that $U$ is conservative (i.e., reflects equivalences) and that $U$ preserves $U$-split geometric realizations (\cite[Theorem~6.2.2.5]{HA}). We will not need to discuss here what $U$-split means exactly since in our cases all geometric realizations will be preserved.

\begin{prop}\label{prop:monadic}
Let $\calc$ be a presentable $\infty$-category and let $\bbT$ be a theory. Then the forgetful functor $\Mod_{\bbT}(\calc)\to\calc$ is monadic and preserves sifted colimits.
\end{prop}
\begin{proof}
We will show that the evaluation $\Fun^\Pi(\bbT,\calc)\to\calc$
is conservative and preserves sifted colimits.
The result then follows immediately from the monadicity theorem.
The fact that the functor is conservative is clear, so it remains to check the sifted colimit condition.
But the inclusion of the finite product preserving functors
\[
\Fun^\Pi(\bbT,\calc)\to\Fun(\bbT,\calc)
\]
preserves sifted colimits by (4) of \cite[Proposition 5.5.8.10]{HA}, and as colimits in functor $\infty$-categories are computed pointwise the evaluation
\[
\Fun(\bbT,\calc)\to\calc
\]
also preserves sifted colimits.
\end{proof}

We will obtain a converse to the previous proposition in the case of the $\infty$-category of spaces; namely, in this case we will identify algebraic theories with certain monads. To this end, note that  
an arbitrary monadic functor $U\colon \Mod_M(\cals) \to\cals$ defines a theory $\bbT_M$ by 
\[\bbT_M:={\big(\Modff_M(\cals)}\big)\op,\]
where $\Modff_M(\cals)\subseteq\Mod_M(\cals)$ is the full subcategory spanned by the free $M$-algebras on finite sets (which we abusively refer to as {\em finite free} algebras, and should not to be confused with more general free algebras on finite or finitely presented spaces).
There is a canonical functor
\[
R\colon\Mod_M(\cals) \to \Mod_{\bbT_M}(\cals)
\]
from modules for $M$ to models to the associated theory $\bbT_M$, which is just the restriction of the Yoneda embedding to the full subcategory $\Modff_M(\cals)$.

\begin{defn}\label{monadalg}
A monadic functor $U\colon\Mod_M(\cals)\to\cals$ is called \emph{algebraic} if
\[
R\colon\Mod_M(\cals) \to \Mod_{\bbT_M}(\cals)
\]
is an equivalence of $\infty$-categories over $\cals$.
We also say that a monad $M$ on spaces is \emph{algebraic} if the associated forgetful functor $U\colon\Mod_M(\cals)\to\cals$ is algebraic.
\end{defn}

The main result of this section is \autoref{algebraicity}, which provides necessary and sufficient conditions for a monadic functor to spaces to be algebraic.
As preparation, we first collect the following result, a straightforward generalization of a well-known result in ordinary category theory.

\begin{prop}\label{prop:modpres}
Let $\calc$ be a presentable $\infty$-category and let $M\colon\calc\to\calc$ be a monad which commutes with $\kappa$-filtered colimits for some infinite regular cardinal $\kappa$. Then $\Mod_M(\calc)$ is a presentable $\infty$-category.
\end{prop}
\begin{proof}
To begin with let us choose a regular cardinal $\kappa$ such that $\calc$ is $\kappa$-compactly generated and $M$ commutes with $\kappa$-filtered colimits. Let $\calc^\kappa\subseteq\calc$ and $\Mod_M(\calc)^\kappa\subseteq\Mod_M(\calc)$ be the respective full subcategories spanned by the $\kappa$-compact objects. We claim that there is an equivalence $\Ind_\kappa(\Mod_M(\calc)^\kappa)\simeq \Mod_M(\calc).$ Since $\Mod_M(\calc)$ admits $\kappa$-filtered colimits, the inclusion $\Mod_M(\calc)^\kappa\subseteq\Mod_M(\calc)$ induces a functor
\[
\phi\colon\Ind_\kappa(\Mod_M(\calc)^\kappa)\to \Mod_M(\calc)
\]
which we want to show is an equivalence. The fully faithfulness of $\phi$ is a special case of the following: if $\cald$ be an $\infty$-category with $\kappa$-filtered colimits, then the inclusion $\cald^\kappa\subseteq\cald$ of the $\kappa$-compact objects induces a fully faithful functor $\Ind_\kappa(\cald^\kappa)\to\cald.$ Thus it remains to show that $\phi$ is essentially surjective.

Because $M$ commutes with $\kappa$-filtered colimits, we see that, if $X\in\calc^\kappa$, then $FX\in\Mod_M(\calc)^\kappa$, where $F\colon\calc\to\Mod_M(\calc)$ denotes a left adjoint to the forgetful functor $\Mod_M(\calc)\to\calc.$
Since the forgetful functor $\Mod_M(\calc)\to\calc$ is conservative and $\calc$ is $\kappa$-compactly generated, a map $f\colon A\to B$ of $M$-modules is an equivalence if and only if
\[
\map_{\Mod_M(\calc)}(FX,A)\to \map_{\Mod_M(\calc)}(FX,B)
\]
is an equivalence for all $X\in\calc^\kappa$.
We will apply this criterion to the map
\[
\colim_{A'\in\Mod_M(\calc)^\kappa_{/A}} A'\to A\, ,
\]
whose domain is a $\kappa$-filtered colimit, in order to obtain the essential surjectivity of~$\phi$.
We first show that, for any $X\in\calc^\kappa$, the induced map
\[
\colim_{\Mod_M(\calc)^\kappa_{/A}}\pi_0\map(FX,A')\to\pi_0\map(FX,A)
\]
is an isomorphism.
Indeed, it is surjective because any (homotopy class of) map $FX\to A$ is the image of the identity map $FX\to A'$ for $A'=FX$, which is by construction a $\kappa$-compact object of $\Mod_M(\calc)$.
Similarly, injectivity follows because given any two maps $f,g:FX\to A'$, the fact that $\Mod_M(\calc)^\kappa_{/A}$ is $\kappa$-filtered implies that there exists an $A''\to A$ which coequalizes $f$ and $g$.
Replacing $X$ by $K\otimes X$ for some finite simplicial set $K$, and noting that $K\otimes X$ is a $\kappa$-compact object of $\calc$ since $K$ is finite, we obtain an isomorphism
\[
\pi_0\map(K,\colim\map(FX,A'))\cong\pi_0\map(K,\map(FX,A)).
\]
It follows that $\colim\map(FX,A')\to\map(FX,A)$ is a homotopy equivalence, as desired.
\end{proof}

\begin{thm}\label{algebraicity}
 A monadic functor $U\colon\Mod_M(\cals)\to\cals$ is algebraic if and only if it preserves sifted colimits. 
\end{thm}
\begin{proof}
Since the forgetful functor $\Mod_{\bbT_M}(\cals)\to\cals$ preserves sifted colimits {(see \autoref{prop:monadic})}, we see that this is a necessary condition.
Thus, suppose that $U$ preserves sifted colimits; we must show that $R$ is an equivalence.
Note that $\Mod_M(\cals)$ is presentable by \autoref{prop:modpres}, and $\Mod_{\bbT_M}(\cals)$ is presentable as an accessible localization of the presentable $\infty$-category $\Fun(\bbT,\cals)$.
Because $\Algff_M(\cals)\subseteq\Mod_M(\cals)$ is a subcategory of compact projective objects, $R$ preserves sifted colimits, and clearly $R$ also preserves small limits. Thus $R$ admits a left adjoint~$L$.

We now check that the adjunction counit $LR\to\id$ is an equivalence. Since $R$ is conservative, as both the projections down to $\cals$ are conservative, this will also imply that the unit $\id\to RL$ is an equivalence.
Observe that both functors commute with sifted colimits and spaces is freely generated under sifted colimits by the finite sets $\langle n\rangle$, it is enough to check the counit equivalence on objects of the form $F\langle n\rangle$. Now, $RF\langle n\rangle=\widehat{F}\langle n\rangle$, the functor represented by $\widehat{F}\langle n\rangle$, so we must show that we have 
an equivalence $L\widehat{F}\langle n\rangle\to F\langle n\rangle$.
Let $A\in\Mod_M(\cals)$ and consider the map
\[\map(F\langle n\rangle,A)\to\map(L\hat{F}\langle n\rangle,A)\, .\]
The left hand side can be identified with $\map(F\langle n\rangle,A)\simeq U(A)^n.$ 
Similarly, the right hand side is
\[\map(L\widehat{F}\langle n\rangle,A)\simeq\map(\widehat{F}\langle n\rangle,RA)\simeq\map(\widehat{F}\langle 1\rangle,RA)^n\simeq U(A)^n\]
where we used in the last step that~$R$ is compatible with the forgetful functors to~$\cals$.
\end{proof}

Finally, we wish to apply the results of this section to the study of semirings and rings in $\infty$-categories. We begin by showing that semirings and rings are algebraic over spaces.

\begin{prop}
The functors $\Rig_{\bbE_n}(\cals) \to \cals$ and $\Ring_{\bbE_n}(\cals) \to \cals$ are monadic and algebraic over $\cals$.
\end{prop}
\begin{proof}
We claim that the functors $\Rig_{\bbE_n}(\cals) \to \Mon(\cals) \to \cals$ and $\Ring_{\bbE_n}(\cals)\to\Grp(\cals) \to\cals$ all preserve sifted colimits and reflect equivalences. Then the monadicity follows from the Barr-Beck theorem \cite[Theorem 6.2.2.5]{HA}, and the algebraicity from \autoref{algebraicity}.

To see that this claim is true note that three of the four functors, namely $\Rig_{\bbE_n}(\cals) \to \Mon(\cals),$ $\Mon(\cals)\to\cals,$ and $\Ring_{\bbE_n}(\cals)\to \Grp(\cals)$, are forgetful functors from $\infty$-categories of algebras over an $\infty$-operad. These forgetful functors are always conservative and for suitable monoidal structures they also preserve sifted colimits \cite[Proposition~3.2.3.1]{HA}. Thus we only have to establish the same properties for $\Grp(\cals) \to \cals$. It is easy to see that this functor is conservative since $\Grp(\cals)$ is a full subcategory of $\Mon(\cals)$ and the given functor factors over the conservative functor $\Mon(\cals)\to\cals.$ It remains to show that $\Grp(\cals) \to \Mon(\cals)$ preserves sifted colimits. But for an $\bbE_\infty$-monoid in the $\infty$-category of spaces being a group object is equivalent to being grouplike. Thus, via the left adjoint functor $\pi_0$ it reduces to the statement that the sifted colimit of groups formed in the 
category of monoids is again a group. And this result is a special case of \cite[Proposition~9.3]{ARV}.
\end{proof}

\begin{defn}
We denote the algebraic theory corresponding to the functor $\Rig_{\bbE_n}(\cals) \to \cals$ by $\bbT_{{\bbE_n}\text{-}\Rig}$ and call it the \emph{theory of $\bbE_n$-semirings}. Accordingly we denote the
algebraic theory corresponding to the functor $\Ring_{\bbE_n}(\cals) \to \cals$ by $\bbT_{{\bbE_n}\text{-}\Ring}$ and call it the \emph{theory of $\bbE_n$-rings}.
\end{defn}

\begin{prop}
Let $\calc$ be a cartesian closed, presentable $\infty$-category. Then we have equivalences
\begin{equation*}
 \Rig_{\bbE_n}(\calc) \simeq \Alg_{\bbT_{{\bbE_n}\text{-}\Rig}}(\calc) \qquad \text{and} \qquad \Ring_{\bbE_n}(\calc) \simeq \Alg_{\bbT_{{\bbE_n}\text{-}\Ring}}(\calc).
\end{equation*}
\end{prop}
\begin{proof}
For $\calc = \cals$ the $\infty$-category of spaces the statement is true by definition of $\bbT_{{\bbE_n}\text{-}\Rig}$ and $\bbT_{{\bbE_n}\text{-}\Ring}$. The general case follows from the base change formulas given in \autoref{prop_rig_ten} and \autoref{prop_ten_al}.
\end{proof}

\begin{rmk}
\begin{enumerate}
\item Theories of $\bbE_\infty$-semirings and rings have also been constructed in \cite{cranch} by the use of spans and distributive laws. These two approaches do agree.
\item The theory approach of semirings and rings gives a way of defining ring objects in a much broader generality. One only needs an $\infty$-category $\calc$ with finite products. In this way we can drop the assumption that
$\calc$ is presentable and cartesian closed. However in this case semiring and ring objects do not admit a nice description in terms of a tensor product on monoids. It is also impossible to apply 
this to different tensor products than the cartesian one.
\item We showed in \autoref{cor:lax} that an accessible, product preserving functor $F\colon \calc \to \cald$ between cartesian closed symmetric monoidal categories extends to a lax symmetric monoidal functor
 $\Mon(\calc) \to \Mon(\cald)$. This means that $F$ extends to functors $\Rig_{\bbE_n}(\calc) \to \Rig_{\bbE_n}(\cald)$ and $\Ring_{\bbE_n}(\calc)\to\Ring_{\bbE_n}(\cald)$. Therefore we may drop the assumption that $F$ is accessible and conclude that any 
product preserving functor $\calc \to \cald$ extends to functors $\Rig_{\bbE_n}(\calc) \to \Rig_{\bbE_n}(\cald)$ and $\Ring_{\bbE_n}(\calc)\to\Ring_{\bbE_n}(\cald)$.
\end{enumerate}
 \end{rmk}

\bibliographystyle{alpha}
\bibliography{add_final}

\end{document}

%% file: decls.tex
\usepackage{amssymb,amsmath,stmaryrd,mathrsfs}
\def\definetac{\newif\iftac}    
\ifx\tactrue\undefined
  \definetac
  \ifx\state\undefined\tacfalse\else\tactrue\fi
\fi
\iftac\else\usepackage{amsthm}\fi
\usepackage[all,2cell]{xy}
\UseAllTwocells
\usepackage{enumitem}
\usepackage{xcolor}
\definecolor{darkgreen}{rgb}{0,0.45,0} 
\usepackage[pagebackref,colorlinks,citecolor=darkgreen,linkcolor=darkgreen]{hyperref}
\usepackage{mathtools}          
\usepackage{tikz}
\usepackage{braket}             

\usepackage{url}                
\usepackage{xspace}             

\makeatletter
\let\ea\expandafter

\def\mdef#1#2{\ea\ea\ea\gdef\ea\ea\noexpand#1\ea{\ea\ensuremath\ea{#2}\xspace}}
\def\alwaysmath#1{\ea\ea\ea\global\ea\ea\ea\let\ea\ea\csname your@#1\endcsname\csname #1\endcsname
  \ea\def\csname #1\endcsname{\ensuremath{\csname your@#1\endcsname}\xspace}}

\DeclareRobustCommand\widecheck[1]{{\mathpalette\@widecheck{#1}}}
\def\@widecheck#1#2{%
    \setbox\z@\hbox{\m@th$#1#2$}%
    \setbox\tw@\hbox{\m@th$#1%
       \widehat{%
          \vrule\@width\z@\@height\ht\z@
          \vrule\@height\z@\@width\wd\z@}$}%
    \dp\tw@-\ht\z@
    \@tempdima\ht\z@ \advance\@tempdima2\ht\tw@ \divide\@tempdima\thr@@
    \setbox\tw@\hbox{%
       \raise\@tempdima\hbox{\scalebox{1}[-1]{\lower\@tempdima\box
\tw@}}}%
    {\ooalign{\box\tw@ \cr \box\z@}}}

\newcount\foreachcount

\def\foreachletter#1#2#3{\foreachcount=#1
  \ea\loop\ea\ea\ea#3\@alph\foreachcount
  \advance\foreachcount by 1
  \ifnum\foreachcount<#2\repeat}

\def\foreachLetter#1#2#3{\foreachcount=#1
  \ea\loop\ea\ea\ea#3\@Alph\foreachcount
  \advance\foreachcount by 1
  \ifnum\foreachcount<#2\repeat}

\def\definescr#1{\ea\gdef\csname s#1\endcsname{\ensuremath{\mathscr{#1}}\xspace}}
\foreachLetter{1}{27}{\definescr}
\def\definecal#1{\ea\gdef\csname c#1\endcsname{\ensuremath{\mathcal{#1}}\xspace}}
\foreachLetter{1}{27}{\definecal}
\def\definebold#1{\ea\gdef\csname b#1\endcsname{\ensuremath{\mathbf{#1}}\xspace}}
\foreachLetter{1}{27}{\definebold}
\def\definebb#1{\ea\gdef\csname l#1\endcsname{\ensuremath{\mathbb{#1}}\xspace}}
\foreachLetter{1}{27}{\definebb}
\def\definefrak#1{\ea\gdef\csname f#1\endcsname{\ensuremath{\mathfrak{#1}}\xspace}}
\foreachletter{1}{9}{\definefrak} 
\foreachletter{10}{27}{\definefrak}
\def\definebar#1{\ea\gdef\csname #1bar\endcsname{\ensuremath{\overline{#1}}\xspace}}
\foreachLetter{1}{27}{\definebar}
\foreachletter{1}{8}{\definebar} 
\foreachletter{9}{15}{\definebar} 
\foreachletter{16}{27}{\definebar}
\def\definetil#1{\ea\gdef\csname #1til\endcsname{\ensuremath{\widetilde{#1}}\xspace}}
\foreachLetter{1}{27}{\definetil}
\foreachletter{1}{27}{\definetil}
\def\definehat#1{\ea\gdef\csname #1hat\endcsname{\ensuremath{\widehat{#1}}\xspace}}
\foreachLetter{1}{27}{\definehat}
\foreachletter{1}{27}{\definehat}
\def\definechk#1{\ea\gdef\csname #1chk\endcsname{\ensuremath{\widecheck{#1}}\xspace}}
\foreachLetter{1}{27}{\definechk}
\foreachletter{1}{27}{\definechk}
\def\defineul#1{\ea\gdef\csname u#1\endcsname{\ensuremath{\underline{#1}}\xspace}}
\foreachLetter{1}{27}{\defineul}
\foreachletter{1}{27}{\defineul}

\def\autofmt@n#1\autofmt@end{\mathrm{#1}}
\def\autofmt@b#1\autofmt@end{\mathbf{#1}}
\def\autofmt@l#1#2\autofmt@end{\mathbb{#1}\mathsf{#2}}
\def\autofmt@c#1#2\autofmt@end{\mathcal{#1}\mathit{#2}}
\def\autofmt@s#1#2\autofmt@end{\mathscr{#1}\mathit{#2}}
\def\autofmt@f#1\autofmt@end{\mathsf{#1}}
\def\autofmt@u#1\autofmt@end{\underline{\smash{\mathsf{#1}}}}
\def\autofmt@U#1\autofmt@end{\underline{\underline{\smash{\mathsf{#1}}}}}
\def\autofmt@h#1\autofmt@end{\widehat{#1}}
\def\autofmt@r#1\autofmt@end{\overline{#1}}
\def\autofmt@t#1\autofmt@end{\widetilde{#1}}
\def\autofmt@k#1\autofmt@end{\check{#1}}

\def\auto@drop#1{}
\def\autodef#1{\ea\ea\ea\@autodef\ea\ea\ea#1\ea\auto@drop\string#1\autodef@end}
\def\@autodef#1#2#3\autodef@end{%
  \ea\def\ea#1\ea{\ea\ensuremath\ea{\csname autofmt@#2\endcsname#3\autofmt@end}\xspace}}
\def\autodefs@end{blarg!}
\def\autodefs#1{\@autodefs#1\autodefs@end}
\def\@autodefs#1{\ifx#1\autodefs@end%
  \def\autodefs@next{}%
  \else%
  \def\autodefs@next{\autodef#1\@autodefs}%
  \fi\autodefs@next}

\DeclareSymbolFont{bbold}{U}{bbold}{m}{n}
\DeclareSymbolFontAlphabet{\mathbbb}{bbold}

\newcommand{\bbone}{\ensuremath{\mathbbb{1}}\xspace}

\mdef\delbar{\overline{\partial}}

\mdef\hf{\textstyle\frac12 }
\mdef\thrd{\textstyle\frac13 }
\mdef\qtr{\textstyle\frac14 }

\newcommand{\op}{^{\mathrm{op}}}

\let\adj\dashv
\SelectTips{cm}{}
\newdir{ >}{{}*!/-10pt/@{>}}    

\mdef\Id{\mathrm{Id}}
\mdef\id{\mathrm{id}}
\alwaysmath{ell}
\alwaysmath{infty}
\alwaysmath{odot}
\def\frc#1/#2.{\frac{#1}{#2}}   
\mdef\ten{\mathrel{\otimes}}

\mdef\sqten{\mathrel{\boxtimes}}

\DeclareRobustCommand\widecheck[1]{{\mathpalette\@widecheck{#1}}}
\def\@widecheck#1#2{%
    \setbox\z@\hbox{\m@th$#1#2$}%
    \setbox\tw@\hbox{\m@th$#1%
       \widehat{%
          \vrule\@width\z@\@height\ht\z@
          \vrule\@height\z@\@width\wd\z@}$}%
    \dp\tw@-\ht\z@
    \@tempdima\ht\z@ \advance\@tempdima2\ht\tw@ \divide\@tempdima\thr@@
    \setbox\tw@\hbox{%
       \raise\@tempdima\hbox{\scalebox{1}[-1]{\lower\@tempdima\box
\tw@}}}%
    {\ooalign{\box\tw@ \cr \box\z@}}}

\DeclareMathOperator\colim{colim}

\DeclareMathOperator\Ho{Ho}

\DeclareMathOperator\Map{Map}

\mdef\we{\overset{\sim}{\longrightarrow}}
\mdef\leftwe{\overset{\sim}{\longleftarrow}}

\def\rightarrowtailfill@{\arrowfill@{\Yright\joinrel\relbar}\relbar\rightarrow}
\newcommand\xrightarrowtail[2][]{\ext@arrow 0055{\rightarrowtailfill@}{#1}{#2}}

\def\twoheadrightarrowfill@{\arrowfill@{\relbar\joinrel\relbar}\relbar\twoheadrightarrow}
\newcommand\xtwoheadrightarrow[2][]{\ext@arrow 0055{\twoheadrightarrowfill@}{#1}{#2}}

\def\slashedarrowfill@#1#2#3#4#5{%
  $\m@th\thickmuskip0mu\medmuskip\thickmuskip\thinmuskip\thickmuskip
   \relax#5#1\mkern-7mu%
   \cleaders\hbox{$#5\mkern-2mu#2\mkern-2mu$}\hfill
   \mathclap{#3}\mathclap{#2}%
   \cleaders\hbox{$#5\mkern-2mu#2\mkern-2mu$}\hfill
   \mkern-7mu#4$%
}
\def\rightslashedarrowfill@{%
  \slashedarrowfill@\relbar\relbar\mapstochar\rightarrow}
\newcommand\xslashedrightarrow[2][]{%
  \ext@arrow 0055{\rightslashedarrowfill@}{#1}{#2}}
\mdef\hto{\xslashedrightarrow{}}
\mdef\htoo{\xslashedrightarrow{\quad}}

\long\def\my@drawfill#1#2;{%
\@skipfalse
\fill[#1,draw=none] #2;
\@skiptrue
\draw[#1,fill=none] #2;
}
\newif\if@skip
\newcommand{\skipit}[1]{\if@skip\else#1\fi}
\newcommand{\drawfill}[1][]{\my@drawfill{#1}}

\newif\ifhyperref
\@ifpackageloaded{hyperref}{\hyperreftrue}{\hyperreffalse}
\iftac
  \let\your@state\state
  \def\state#1{\gdef\currthmtype{#1}\your@state{#1}}
  \let\your@staterm\staterm
  \def\staterm#1{\gdef\currthmtype{#1}\your@staterm{#1}}
  \let\defthm\newtheorem
  \def\currthmtype{}
  \ifhyperref
    \def\autoref#1{\ref*{label@name@#1}~\ref{#1}}
  \else
    \def\autoref#1{\ref{label@name@#1}~\ref{#1}}
  \fi
  \AtBeginDocument{%
    \let\old@label\label%
    \def\label#1{%
      {\let\your@currentlabel\@currentlabel%
        \edef\@currentlabel{\currthmtype}%
        \old@label{label@name@#1}}%
      \old@label{#1}}
  }
\else
  \ifhyperref
    \def\defthm#1#2{%
      \newtheorem{#1}{#2}[section]%
      \expandafter\def\csname #1autorefname\endcsname{#2}%
      \expandafter\let\csname c@#1\endcsname\c@thm}
  \else
    \def\defthm#1#2{\newtheorem{#1}[thm]{#2}}
    \ifx\SK@label\undefined\let\SK@label\label\fi
    \let\old@label\label
    \let\your@thm\@thm
    \def\@thm#1#2#3{\gdef\currthmtype{#3}\your@thm{#1}{#2}{#3}}
    \def\currthmtype{}
    \def\label#1{{\let\your@currentlabel\@currentlabel\def\@currentlabel%
        {\currthmtype~\your@currentlabel}%
        \SK@label{#1@}}\old@label{#1}}
    \def\autoref#1{\ref{#1@}}
  \fi
\fi

\newtheorem{thm}{Theorem}[section]

\defthm{cor}{Corollary}
\defthm{prop}{Proposition}
\defthm{lem}{Lemma}
\defthm{sch}{Scholium}
\defthm{assume}{Assumption}
\defthm{claim}{Claim}
\defthm{conj}{Conjecture}
\defthm{hyp}{Hypothesis}
\defthm{fact}{Fact}

\iftac\theoremstyle{plain}\else\theoremstyle{definition}\fi
\defthm{defn}{Definition}
\defthm{notn}{Notation}
\defthm{rmk}{Remark}
\defthm{eg}{Example}
\defthm{egs}{Examples}
\defthm{warn}{Warning}
\defthm{ceg}{Counterexample}

\iftac\theoremstyle{plain}\else\theoremstyle{remark}\fi
\defthm{ex}{Exercise}

\def\thmqedhere{\expandafter\csname\csname @currenvir\endcsname @qed\endcsname}

\setitemize[1]{leftmargin=2em}
\setenumerate[1]{leftmargin=*}

\iftac
  \let\c@equation\c@subsection
\else
  \let\c@equation\c@thm
\fi
\numberwithin{equation}{section}

\@ifpackageloaded{mathtools}{\mathtoolsset{showonlyrefs,showmanualtags}}{}

\alwaysmath{alpha}
\alwaysmath{beta}
\alwaysmath{gamma}
\alwaysmath{Gamma}
\alwaysmath{delta}
\alwaysmath{Delta}
\alwaysmath{epsilon}
\mdef\ep{\varepsilon}
\alwaysmath{zeta}
\alwaysmath{eta}
\alwaysmath{theta}
\alwaysmath{Theta}
\alwaysmath{iota}
\alwaysmath{kappa}
\alwaysmath{lambda}
\alwaysmath{Lambda}
\alwaysmath{mu}
\alwaysmath{nu}
\alwaysmath{xi}
\alwaysmath{pi}
\alwaysmath{rho}
\alwaysmath{sigma}
\alwaysmath{Sigma}
\alwaysmath{tau}
\alwaysmath{upsilon}
\alwaysmath{Upsilon}
\alwaysmath{phi}
\alwaysmath{Pi}
\alwaysmath{Phi}
\mdef\ph{\varphi}
\alwaysmath{chi}
\alwaysmath{psi}
\alwaysmath{Psi}
\alwaysmath{omega}
\alwaysmath{Omega}

\let\Th\Theta

\makeatother
